\documentclass[11pt]{amsart}
\usepackage{amssymb, amsfonts, amscd, enumerate}

\usepackage{amsmath}
\usepackage{graphicx, enumerate}

\usepackage{verbatim}
\usepackage{amscd}
\usepackage{xypic}
\usepackage{amsmath, amssymb}
\usepackage{pstricks}
\usepackage{pst-plot}

\theoremstyle{plain}
\newtheorem{theorem}{Theorem}[section]
\newtheorem{lemma}[theorem]{Lemma}
\newtheorem{proposition}[theorem]{Proposition}
\newtheorem{corollary}[theorem]{Corollary}

\theoremstyle{definition}
\newtheorem{definition}[theorem]{Definition}
\newtheorem{remark}[theorem]{Remark}
\newtheorem{remarks}[theorem]{Remarks}

\numberwithin{equation}{section}

\newcommand{\et}{\operatorname{et}}
\newcommand{\gr}{\operatorname{gr}}
\newcommand{\id}{\operatorname{id}}

\newcommand{\st}{\operatorname{s}}
\newcommand{\sst}{\operatorname{ss}}

\newcommand{\bfd}{{\bf d}}
\newcommand{\bfe}{{\bf e}}
\newcommand{\bfdim}{{\bf dim}}
\newcommand{\bfi}{{\bf i}}
\newcommand{\bfI}{{\bf I}}

\newcommand{\Aut}{\operatorname{Aut}}
\newcommand{\End}{\operatorname{End}}
\newcommand{\GL}{\operatorname{GL}}
\newcommand{\Gr}{\operatorname{Gr}}
\newcommand{\Hom}{\operatorname{Hom}}
\newcommand{\Ker}{\operatorname{Ker}}
\newcommand{\Ima}{\operatorname{Im}}

\newcommand{\Proj}{\operatorname{Proj}}
\newcommand{\PGL}{\operatorname{PGL}}

\newcommand{\Spec}{\operatorname{Spec}}
\newcommand{\Tr}{\operatorname{Tr}}

\newcommand{\bF}{\mathbb F}
\newcommand{\bG}{\mathbb G}
\newcommand{\bN}{\mathbb N}
\newcommand{\bP}{\mathbb P}
\newcommand{\bQ}{\mathbb Q}

\newcommand{\bZ}{\mathbb Z}

\newcommand{\cF}{\mathcal F}

\newcommand{\cM}{\mathcal M}

\newcommand{\cO}{\mathcal O}

\numberwithin{equation}{subsection}
\newcommand{\be}
  {\protect\setcounter{equation}{\value{subsubsection}}}  
  \newcommand{\ee}%
   {\protect\setcounter{subsubsection}{\value{equation}}}

  {\protect\setcounter{subsubsection}{\value{equation}}}

\def \C{\mathcal C}

\def \cL{\mathcal L}

\def \colimn{\underset {n \rightarrow \infty}  {\hbox {lim}}}

\def \colimK.{\underset {\underset K^.  \rightarrow}  {\hbox {lim}}}

\def \colimU.{\underset {\underset U_.  \rightarrow}  {\hbox {lim}}}

\def \EGx{EG{\underset G  \times}}

\def \EG1{E{(G \times {\mathbb C}^*)}{\underset {G\times {\mathbb C}^*}  \times}}

\def \EZ(s)1{E{(Z(s) \times {\mathbb C}^*)}{\underset {(Z(s)\times {\mathbb C}^*)}  \times}}

\def \EM(u){EM(u){\underset {M(u)}  \times}}
\def \EM(us){EM(u,s){\underset {M(u, s)}  \times}}

\def \invlim1{\underset {\infty \leftarrow q}  {\hbox {lim}}^1}

\def \L3{\Lambda \times \Lambda \times \Lambda}
\def \L2{\Lambda \times \Lambda}

\def \longright2arrow{{\overset \longrightarrow  {\overset {}  \longrightarrow}}}

\def \leq{\le}

\def \O{{\mathcal O}}

\def \Q{\mathbb Q}

\def \ra{\rightarrow}
\def \Ra{\Rightarrow}
\def \RG^{R(G)^{\hat {}}\ }
\def \resp{respectively}
\def \res{respectively}

\def \topGcoh*{^{top, *} _{G}}
\def \topGho*{ _{top,*} ^{G}}

\def \Z(s){Z(s) \times {\mathbb C}^*}
\def \Z{\mathbb Z}

\author{M.~Brion and R.~Joshua}
\address{Michel Brion\\
Universit\'e de Grenoble--I\\
Institut Fourier, BP 74\\ 
38402 Saint-Martin d'H\`eres Cedex\\
France}
\email{Michel.Brion@ujf-grenoble.fr}

\address{Roy Joshua\\ 
Department of Mathematics\\  
Ohio State University\\
231 W 18th Avenue\\
Columbus, OH 43210\\
USA}
\email{joshua@math.ohio-state.edu}

\begin{document}

\begin{abstract}
We explore several variations of the notion of {\it purity} for the action of Frobenius on
schemes defined over finite fields. In particular, we study how these notions are preserved
under certain natural operations like quotients for principal bundles and also geometric 
quotients for reductive group actions. We then apply these results to study the cohomology 
of quiver moduli. We prove that a natural stratification of the space of representations 
of a quiver with a fixed dimension vector is equivariantly perfect and from it deduce that 
each of the $l$-adic cohomology groups of the quiver moduli space is strongly pure.
\end{abstract}

\title[Purity and quiver moduli]
{Notions of purity and the cohomology of quiver moduli}

\date{}

\thanks{The second author thanks the Institut Fourier, the MPI (Bonn), the IHES (Paris) 
and the NSA for support}
\maketitle

\setcounter{section}{-1}

\section{\bf Introduction}
\label{sec:introduction}

Consider a scheme $X$ of finite type over a finite field $\bF_q$. Then 
the number of points of $X$ that are rational over a finite extension 
$\bF_{q^n}$ is expressed by the trace formula
\begin{equation}
\label{eqn:trace}
|X(\bF_{q^n})| = \sum _{i \ge 0} (-1)^i \Tr\big( F^n, H^i_c(\bar X, \bQ_l) \big)
\end{equation} 
where $\bar X = X {\underset {\Spec \bF_q} \times }\Spec \bar{\bF}_q$ and $F$ 
denotes the Frobenius morphism of $\bar X$. 
The results of Deligne (see \cite{De74a, De77, De80}) show that every 
eigenvalue of $F$ on $H^i_c(\bar X, \bQ_l)$ has absolute value $q^{w/2}$ 
for some non-negative integer $w \le i$.

\medskip

In \cite{BP10}, the first named author and Peyre studied in detail the properties 
of the counting function $n \mapsto |X(\bF_{q^n} )|$, when $X$ is a homogeneous 
variety under a linear algebraic group (all defined over $\bF_q$). 
For this, they introduced the notion of a 
weakly pure variety $X$, by requesting that all eigenvalues of $F$ in 
$H^*_c(\bar X, \bQ_l)$ are of the form $\zeta q^j$, where $\zeta$ is a root
of unity, and $j$ a non-negative integer. This implies that the counting 
function of $X$ is a periodic polynomial with integer coefficients, i.e., 
there exist a positive integer $N$ and polynomials 
$P_0(t),\ldots,P_{N-1}(t)$ in $\bZ[t]$ such that 
$\vert X(\bF_{q^n} \vert = P_r(q^n)$ whenever $n \equiv r$ (mod $N$). They 
also showed that homogeneous varieties under linear algebraic groups are 
weakly pure.

\medskip

The present paper arose out of an attempt to study the notion of weak purity
in more detail, and to see how it behaves with respect to torsors and 
geometric invariant theory quotients. While applying this notion
to moduli spaces of quiver representations, it also became clear that they 
satisfy a stronger notion of purity, which in fact differs from the notion
of a strongly pure variety, introduced in \cite{BP10} as a technical device.   
Thus, we were led to define weak and strong purity in a more general 
setting, and to modify the notion of strong purity so that it applies to 
GIT quotients.

\medskip

Here is an outline of the paper. In the first section, we introduce 
a notion of weak purity for equivariant local systems (generalizing that in 
\cite{BP10} where only the constant local system is considered) and a closely 
related notion of strong purity.  
The basic definitions are in Definition~\ref{def.weak-strong-purity}. 
Then we study how these notions behave with respect to torsors and certain 
associated fibrations. The main result is the following.

\begin{theorem} (See Theorem~\ref{thm:main1}.)
Let $\pi : X \to Y$ denote a torsor under a linear algebraic group $G$, 
all defined over $\bF_q$.
Let $\C_X$ denote a class of $G$-equivariant $l$-adic local systems on $X$, and $\C_Y$ a 
class of $G$-equivariant $l$-adic local systems on $Y$, where $Y$ is provided with the trivial
$G$-action.
\medskip

(i) Suppose $\C_X \supseteq \pi^*(\C_Y)$.
Then if $X$ is  weakly pure with respect to $\C_X$  so is $Y$
with respect to $\C_Y$. In case  $G$ is split, and if  $X$ is strongly pure with 
respect to $\C_X$, then so is $Y$ with respect to $\C_Y$.  
\medskip

(ii) Suppose $G$ is also connected and $\C_X= \pi^*(\C_Y)$. Then if $Y$ is weakly pure with
respect to $\C_Y$, so is $X$ with respect to $\C_X$. In case $G$ is split and if $Y$ is strongly pure
with respect to $\C_Y$, so is $X$ with respect to $\C_X$.

\end{theorem}

As explained just after Theorem \ref{thm:main1}, the above theorem applies 
with the following choice of the classes $\C_X$ and $\C_Y$: 
let $\C_X$ denote the class of $G$-equivariant $l$-adic local systems on $X$ 
obtained as split summands of $\rho_{X*}({\underline \Q}_l^{\oplus ^n})$ 
for some $n >0$, where $\rho_X: X' \ra X$ is a $G$-equivariant finite \'etale 
map and ${\underline \Q}_l^{\oplus ^n}$ is the constant local system of rank $n$ 
on $X'$, and define $\C_Y$ similarly.

In the second section, we first apply some of the above results to show that geometric 
invariant theory quotients of smooth varieties by connected reductive groups 
preserve the properties of weak and strong purity. Specifically, we obtain 
the following.

\begin{theorem} (See Theorem ~\ref{thm:main2}.)
Consider a smooth variety $X$ provided with the action of a connected reductive
group $G$ and with an ample, $G$-linearized line bundle $L$, such that the
following two conditions are satisfied:

\medskip

\noindent
{\rm (i)} Every semi-stable point of $X$ with respect to $L$ is stable. 

\medskip

\noindent
{\rm (ii)} $X$ admits an equivariantly perfect stratification 
with open stratum the (semi)-stable locus.
 
\medskip

If $X$ is weakly pure with respect to the constant local system $\bQ_l$, 
then so is the geometric invariant theory quotient $X/\!/G$. 

\medskip

If $G$ is split and $X$ is strongly pure with respect to $\bQ_l$, 
then so is $X/\!/G$.
\end{theorem}

Here we recall that a stratification by smooth, locally closed $G$-subvarieties 
is {\it equivariantly perfect}, if the associated long exact sequences 
in equivariant cohomology break up into short exact sequences
(see e.g. \cite[p.~34]{Kir84}. When $X$ is projective, such a 
stratification has been constructed by Kirwan via an analysis of semi-stability
(see \cite[Theorem 13.5]{Kir84}, and also the proof of Corollary \ref{cor:git}
for details on equivariant perfection). 
We also obtain an extension of the above theorem with local systems in the 
place of $\bQ_l$: this is in Theorem~\ref{thm:main3}.

Next, we study in detail the quiver moduli spaces using these techniques. 
In particular, we show that the space of representations of a given quiver 
with a fixed dimension vector satisfies our assumption (ii) (since that
space is affine, this assertion does not follow readily from Kirwan's theorem 
quoted above). Our main result in this setting is the following. 

\begin{theorem} (See Theorem ~\ref{thm:main4}.)
Let $X$ denote the representation space of a given quiver with a given dimension 
vector. Then the stratification of $X$ defined by using semi-stability with respect 
to a fixed character $\Theta$ is equivariantly perfect.
\end{theorem}

Thus, the condition (ii) in Theorem \ref{thm:main2} holds in this setting.
Since the condition (i) holds for general values of $\Theta$, it follows that 
the corresponding geometric invariant theory quotient (i.e. the quiver moduli space) 
is strongly pure with respect to $\bQ_l$. i.e. We obtain the following result (see Corollary ~\ref{cor.quiver.moduli}):
\begin{corollary} 
Assume in addition to the above situation that each semi-stable point is stable.
Then the $l$-adic cohomology $H^*\big( M^{\Theta-\st}(Q, \bfd),\bQ_l \big)$ 
vanishes in all odd degrees, $H^*\big( M^{\Theta-\st}(Q, \bfd),\bQ_l \big)$ is strongly pure,
and hence the number of $\bF_{q^n}$-rational points of $M^{\Theta-\st}(Q, \bfd)$
is a polynomial function of $q^n$ with integer coefficients.
\end{corollary} 
\vskip .3cm
The last corollary recovers certain results of Reineke 
(see \cite[Section 6]{Re03} and \cite[Theorem 6.2]{Re06}) which are established 
by using the combinatorics of the Hall algebra associated to the quiver. 
Our proof is purely based on geometric invariant theory coupled with the theory 
of weak and strong purity developed in the first section of this paper.

\medskip

It may be worth pointing out that several of the varieties that are weakly pure, 
for example, connected reductive groups, turn out to be mixed Tate. The weight 
filtration and the {\it slice-filtration} for such varieties are related in 
\cite{HK06}. In view of this, we hope to explore the results of this paper 
in a motivic context in a sequel. 

\medskip

{\bf Acknowledgments}. The authors would like to thank the referees for their 
valuable remarks and comments.

\section{\bf Notions of purity}

\subsection{Equivariant local systems}

Throughout the rest of the paper, we will only consider separated schemes of 
finite type which are defined over a finite field $\bF=\bF_q$ with $q$ elements 
where $q$ is a power of the characteristic $p$.  Given such a scheme $X$, 
we will denote by $X_{\bar \bF}$ (or $\bar X$) its base extension to the algebraic 
closure $\bar \bF$ of $\bF$, and by $F: \bar X \to \bar X$ the Frobenius morphism.

We will consider $l$-adic sheaves and local systems, where $l\neq p$ is a prime 
number. Recall the following from \cite[Chapitre 5]{BBD81}. Given a scheme $X$, 
an $l$-adic sheaf $\cL = \{ \cL_{\nu} \mid \nu \in \bZ_{\ge 1} \}$ 
on the \'etale site $X_{\et}$ will mean an inverse system of sheaves with each 
$\cL_{\nu}$ a constructible sheaf of $\bZ/l^{\nu}$-modules. Such an $l$-adic sheaf
defines by base
extension an $l$-adic sheaf $\bar \cL$ on $\bar X$ provided with an isomorphism 
$F^*(\bar \cL) \to \bar \cL$. An $l$-adic local system on $X$ is an $l$-adic 
{\it lisse} sheaf $\{\cL_{\nu}|\nu\}$, that is, each $\cL_{\nu}$ is locally constant 
on $X_{\et}$.

We now recall some basic properties of higher direct images of $l$-adic sheaves 
under fibrations. We say that a morphism of schemes $f: X \to Y$ is a 
{\it locally trivial fibration} 
if $f$ is smooth, and there exists an \'etale covering 
$Y' \to Y$ such that the pull-back morphism 
$f' : X' := X {\underset Y \times} Y' \to Y'$ 
is isomorphic to the projection $Y' \times Z \to Y'$ for some scheme $Z$.
Then $Z$ is smooth, and all fibers of $f$ at $\bar{\bF}$-rational points are 
isomorphic to $Z_{\bar{\bF}}$.

The main importance of this notion comes from the following result, which is in
fact a direct consequence of the definitions, the K\"unneth formula in \'etale
cohomology, and \cite[5.1.14]{BBD81}. 

\begin{proposition}\label{prop:fibration}
Let $f: X \to Y$ denote a locally trivial fibration. Let $\nu >0$ denote  any fixed
integer and let $\cF$ denote a constructible sheaf of $\bZ/l^{\nu}$-modules on 
$X_{\et}$ which is constant on some Galois covering of $X$. 

\medskip

{\rm (i)} Then each $R^mf_*(\cF)$ is a constructible sheaf of $\Z/l^{\nu}$-modules on 
$Y_{\et}$.

\medskip

{\rm (ii)} Let $\bar y$ denote a fixed geometric point of $Y$ and let 
$X_{\bar y} = X{\underset Y \times} \bar y$ denote the corresponding geometric fiber. Then
$R^mf_*(\cF)_{\bar y} \simeq H^m(X_{\bar y}, \cF_{|X_{\bar y}})$ for all $y$. In particular,
if $\cF$ is a locally constant constructible sheaf of $\bZ/l^{\nu}$-modules on $X_{\et}$,
then the sheaves $R^mf_*(\cF)$ for all $m \ge 0$ are locally constant constructible 
sheaves on $Y_{\et}$. 

\medskip

{\rm (iii)} Moreover, if $\cL$ is an $l$-adic local system that is mixed and of weight 
$\ge w$, then each $R^mf_*(\cL)$ is also mixed and of weight $\ge m+w$.
\end{proposition}

We now turn to equivariant local systems; for this, we first fix notations and 
conventions about algebraic groups and their actions.
A smooth group scheme will be called an {\it algebraic group}; we will only consider
{\it linear} algebraic groups in this paper. For such a group $G$, we denote by $G^o$
its {\it neutral component}, that is, the connected component containing the 
identity element $e_G$. We recall that $G^o$ is a closed normal subgroup of $G$, and 
that $G/G^o$ is a finite group. 

A scheme $X$ provided with the action of an algebraic group $G$ will be called a 
{\it $G$-scheme}. The action morphism $G \times X \to X$ will be denoted by
$(g,x) \mapsto g \cdot x$.

Given a $G$-scheme $X$, a $G$-{\it equivariant} $l$-adic local system on $X$ will 
denote an $l$-adic sheaf 
$$
\{\cL(m)|m \ge 0\} = \{\cL_{\nu}(m)|\nu \in \bZ_{\ge 1}, m \ge 0\}
$$ 
on the simplicial scheme $\EGx X$ such that the following two conditions are satisfied:

\noindent
(i) Each $\cL_{\nu}(m)$ is a locally constant sheaf of $\bZ/l^{\nu}$-modules 
on the \'etale site 
$$
(\EGx X)(m)_{\et} = (G^m \times X)_{\et}.
$$ 

\noindent
(ii) For any $\nu \ge 0$ and for any structure map $\alpha: [m] \to [n]$ in $\Delta$, 
the induced map $\alpha^*\big( \cL_{\nu}(n) \big) \to \cL_{\nu}(m)$ is an isomorphism.

\medskip

We say that $\cL$ is mixed and of weight $w$ (mixed and of weight $\ge w$), if
the $l$-adic sheaf $\{\cL_{\nu}(m)|\nu \}$
on $(\EGx X)_m$ is mixed and of weight $w$ (mixed and of weight $\ge w$, \res) 
for each $m \ge 0$.  (Since $G$ is smooth, it suffices 
to verify this condition just for $m = 0$.)  Given an $l$-adic local system 
$\cL =\{\cL_{\nu} | \nu \}$ as above, we let $\bar \cL $ denote the corresponding $l$-adic 
local system on $\bar X$ obtained by base extension.  

\begin{remark}
Proposition \ref{prop:fibration} (iii) justifies the condition on the weights 
put into the definition of the class of local systems. In this paper we only consider 
the derived direct image functors and not the derived direct image functors with 
proper supports. The latter functors send complexes of $l$-adic sheaves that are mixed 
and of weight $\le w$ to complexes of $l$-adic sheaves that are mixed and of weight 
$\le w$. Therefore, to consider these functors or $l$-adic cohomology with proper 
supports, one needs to consider classes of local systems that are mixed and of weight 
$\le w$ for some positive integer $w$.
\end{remark}

We now relate equivariant local systems for the actions of a connected 
algebraic group $G$ and of various subgroups. Recall that $G$ contains
a Borel subgroup $B$ (defined over $\bF$) which in turn contains a
maximal torus $T$ (also defined over $\bF$). Moreover, the pairs
$(B,T)$ as above are all conjugate under the group $G(\bF)$.

\begin{proposition}\label{prop:equi} 
Let $X$ denote a $G$-scheme, where $G$ is a connected algebraic group.
Fix a Borel subgroup $B \subseteq G$ and a maximal torus $T \subseteq B$.
Then the restriction functors induce equivalences of categories

($G$-equivariant $l$-adic local systems on $X$) $\simeq$ 
($B$-equivariant $l$-adic local systems on $X$) 
\newline 
$\simeq $ ($T$-equivariant $l$-adic local systems on $X$),

\noindent
where all local systems are considered on the corresponding schemes defined over $\bar{\bF}$.
\end{proposition}
\begin{proof} 
Since $G$ is assumed to be connected, each connected component of $X$ is stable by $G$, so that
we may assume $X$ is also connected.
We will show below that there exists a fibration sequence of \'etale topological types
in the sense of \cite[Chapter 10]{Fr83}:
\be \begin{align}
     \label{fibr.1}
    (G/B)_{\et} \ra (EB{\underset B \times}X)_{\et} &{\overset {\alpha} \ra} (EG{\underset G \times}X)_{\et},\\
    (B/T)_{\et} \ra (ET{\underset T \times} X)_{\et} &{\overset {\beta} \ra} (EB {\underset B \times}X)_{\et} \notag
\end{align} \ee
where the subscript $et$ denotes the \'etale topological types. Recall this means the maps
$\alpha$ and $\beta$ are maps of an inverse system of pointed simplicial sets and that the
homotopy fiber of $\alpha$ (resp. $\beta$) is weakly equivalent in the sense of \cite[Definition 6.1]{Fr83}
to $(G/B)_{\et}$ (resp. $(B/T)_{\et}$). Assuming this, one obtains the exact sequences of pro-groups:
\vskip .3cm
$ \pi_1((G/B)_{\et}) \ra \pi_1((EB{\underset B \times}X)_{\et}) \ra \pi_1(EG{\underset G \times}X)_{\et}) \ra 1$ and
\vskip .3cm
$  \pi_1((B/T)_{\et}) \ra \pi_1((ET{\underset T \times}X)_{\et}) \ra \pi_1(EB{\underset B\times}X)_{\et}) \ra 1$.
\vskip .3cm \noindent
Since the flag variety $G/B$ is the same for the reductive quotient
$G/R_u(G)$, we may assume $G$ is reductive when considering $G/B$.  Therefore, $G/B$  lifts to characteristic
$0$ so that one may see $\pi_1((G/B)_{\et})\,{\widehat {}}=0$ where the completion is away from $p$.
Next observe that $B/T$ is an affine space. Therefore, its \'etale topological type completed
away from $p$ is trivial: this follows from the observation that the \'etale cohomology
of affine spaces are trivial with respect to locally constant torsion sheaves, with torsion prime
 to the
characteristic. (See \cite[Chapter Vi, Corollary 4.20]{Mi80}.)
Applying completion away from  $p$ to the above exact sequences and observing that
such a completion is right-exact (see for example, \cite[Chapter 3, 8.2]{BK72}) and that $\pi_1((G/B)_{\et}){\widehat {}}$ \, and 
$\pi_1((B/T)_{ \et}) {\widehat {}}$\,  are trivial provides the isomorphisms:
$$
\pi_1((EB{\underset B \times}X)_{\et})\, {\widehat {}} \simeq 
\pi_1((EG{\underset G \times}X)_{\et})\, {\widehat {}},
$$
$$
\pi_1((ET{\underset T \times} X)_{\et})\, {\hat{}} \simeq 
\pi_1(( EB {\underset B \times}X)_{\et})\, {\hat{}}.
$$ 
\vskip .3cm
Since $l$-adic equivariant local systems correspond to continuous $l$-adic 
representations of the above completed fundamental groups (see, for example, 
\cite[Appendix (A.3.3)]{Jo93}), the statements in the proposition follow. 
\vskip .3cm
Now we proceed to show that the first fibration sequences of \'etale topological types as in ~\eqref{fibr.1} exists assuming that one has a weak-equivalence:
$(EB{\underset B \times}X)_{\et} \simeq (EG{\underset G \times} G{\underset B \times} X)_{\et}$.
For this we observe first that  the inclusion $X \ra G{\underset B \times}X$ given by sending $x \mapsto (e, x)$
composed with the map $\pi: G{\underset B \times}X \ra X$ given by $(g, x) \mapsto gx$ is the identity.
Therefore, one may readily show that the homotopy fiber of the map $\pi_{\et}: (G{\underset B \times}X)_{\et} \ra X_{\et}$
is given by $(G/B)_{\et}$. Next we observe that 
\vskip .3cm
$(EG{\underset G \times}G{\underset B \times}X)_n = G^n \times  G{\underset B \times}X$ and that $( EG{\underset G \times}X)_n = G^n \times X$.
\vskip .3cm \noindent
Let $p:EG{\underset G \times}G{\underset B \times}X \ra EG{\underset G \times}X$ be the map induced by
$\pi$. Then the homotopy fiber of the map $(p_n)_{\et}$ also identifies with $(G/B)_{\et}$ for all $n \ge 0$.
(This follows from a K\"unneth formula for the \'etale topological types which may be readily deduced from 
\cite[Theorem 10.7]{Fr83}.)
Next observe (see \cite[Lemma 5.2]{Wa78}) that given a diagram of bi-simplicial sets 
$F_{\bullet, \bullet} \ra E_{\bullet, \bullet} \ra B_{\bullet, \bullet}$
so that for each fixed $n$, $F_{\bullet, n} \ra E_{\bullet, n} \ra B_{\bullet, n}$ is a fibration sequence
up to homotopy and  $B_{\bullet, n}$ is connected for each $n$, then the diagram 
$\Delta F_{\bullet, \bullet} \ra \Delta E_{\bullet, \bullet} \ra \Delta B_{\bullet, \bullet}$ is also 
a fibration sequence up to homotopy. Since each $(EG{\underset G \times}X)_n = G^n \times X$ is connected, so
is $\pi(U_{\bullet, n})$ where $U_{\bullet, \bullet}$ is an \'etale  hypercovering of $EG{\underset G \times}X$ and $\pi$ denotes the functor of connected components.
Therefore, one may apply  this with $B_{\bullet, \bullet}$ denoting the
bi-simplicial sets forming $(EG{\underset G \times}X)_{\et}$ and $E_{\bullet, \bullet}$ denoting the bi-simplicial
sets forming $(EG{\underset G \times} G{\underset B \times} X)_{\et}$. It follows that one obtains
the first fibration sequence of \'etale topological types appearing in ~\eqref{fibr.1} assuming that
one has a weak-equivalence 
$(EB{\underset B \times}X)_{\et} \simeq (EG{\underset G \times} G{\underset B \times} X)_{\et}$.
Next we proceed to sketch the existence of such a weak-equivalence.
\vskip .3cm
For this one needs the intermediary simplicial scheme $E(G\times B){\underset {(G\times B)} \times} (G \times X)$
where $G\times B $ acts on $G \times X$ by $(g_1, b_1)\circ (g, x) = (g_1gb_1^{-1}, b_1x)$. This simplicial
scheme maps to both $EB{\underset B \times}X$ and $EG{\underset G \times} G {\underset B \times}X$. 
The simplicial geometric fibers of the map to $EB{\underset B \times}X$ identify with $EG$ and the
simplicial geometric fibers of the map to $EG{\underset G \times} G {\underset B \times}X$ identify with
$EB$. Moreover an argument as in the last paragraph will show that the homotopy fibers of the corresponding
maps of the \'etale topological types are $(EG)_{\et}$ and $(EB)_{\et}$, \resp: clearly these are contractible
so that one obtains a weak-equivalence between
$(EB{\underset B \times}X)_{\et}$, $(E(G\times B){\underset {G \times B} \times} (G \times X))_{\et}$ and
$(EG{\underset G \times} G{\underset B \times X})_{\et}$. This completes the proof of the existence of the
first fibration sequence of \'etale topological types appearing in ~\eqref{fibr.1}.
A similar argument with $G$ (resp. $B$) replaced by $B$ (resp. $T$) proves the second fibration sequence 
of \'etale topological types appearing in ~\eqref{fibr.1} and completes the proof of the Proposition.
\end{proof} 

\subsection{Torsors}

Recall that a {\it torsor under an algebraic group} $G$ (also called a 
{\it principal $G$-bundle}) 
consists of a $G$-scheme $X$ together with a faithfully flat, $G$-invariant 
morphism of schemes $\pi : X \to Y$ such that the morphism 
$$
G \times X \longrightarrow X {\underset Y \times} X, 
\quad (g,x) \longmapsto (g \cdot x, x)
$$
is an isomorphism. Then $Y = X/G$ as topological spaces, and $\pi$ is the
quotient map. Moreover, $\pi$ is a locally trivial fibration with fiber $G$ 
(as follows from \cite[Lemme XIV 1.4]{Ra70}, since $G$ is assumed to be
linear). 

\begin{lemma}\label{lem:desc}
 Let $\pi:X \ra Y$ denote a torsor under an algebraic group $G$.
Then $\pi^*$ induces an equivalence of categories:
\vskip .3cm
($l$-adic local systems on $Y) \simeq (G$-equivariant $l$-adic local systems on $X$). 
\end{lemma}
\begin{proof} The conditions for a sheaf to be $G$-equivariant correspond to descent data
for the map $\pi$. Therefore, the conclusion follows by descent theory.
\end{proof}

We will need the following result, which is probably known but for which we could not 
locate any reference.

\begin{lemma}
\label{lem:tor}
Let $\pi : X \to Y$ denote a torsor under an algebraic group $G$, and 
$H \subset G$ a closed subgroup. Then $\pi$ factors uniquely as 
$$
\CD X @>{\varphi}>> Z @>{\psi}>> Y = X/G \endCD
$$ 
where $\varphi$ is an $H$-torsor, and $\psi$ is a locally trivial fibration
with fiber $G/H$.
 
If $H$ is a normal subgroup of $G$, then the quotient algebraic group 
$G/H$ acts on $Z$, and $\psi$ is a $G/H$-torsor. 

In particular, there is a unique factorization  
$$
\CD X @>{\varphi}>> Z = X/G^o @>{\psi}>> Y = X/G \endCD
$$ 
where $\varphi$ is a $G^o$-torsor and $\psi$ is a Galois cover with group $G/G^o$.
\end{lemma}

\begin{proof}
For the uniqueness of the factorization, note that $Z = X/H$ as topological
spaces, and the structure sheaf $\cO_Z$ is the subsheaf of $H$-invariants 
in $\varphi_*\cO_X$.

For the existence, we may assume that $Y$ is affine: indeed, if we have 
factorizations $X_i \to Z_i \to Y_i$ where $(Y_i)$ is a covering of $Y$ by open 
affine subschemes, then they may be glued to a global factorization, 
by uniqueness. Now $X$ is affine, since so is the morphism $\pi$. Thus,
the product $X \times G/H$ is quasi-projective. Moreover, the projection 
$p_2 : X \times G/H \to X$ admits a relatively ample $G$-linearized 
invertible sheaf (indeed, by a classical theorem of Chevalley, 
the homogeneous space $G/H$ is isomorphic to a $G$-orbit in the 
projectivization $\bP(V)$, where $V$ is a finite-dimensional rational 
representation of $G$. So the pull-back of the invertible sheaf 
$\cO_{\bP(V)}(1)$ yields an ample, $G$-linearized invertible sheaf on $G/H$). 
By \cite[Proposition 7.1]{MFK94}, it follows that the quotient morphism 
$\varpi : X \times G/H \to X {\underset G \times} G/H$ 
exists and is a $G$-torsor, where $X  {\underset G \times} G/H$ is a 
quasi-projective scheme. Moreover, the $G$-equivariant morphism $p_2$ descends 
to a morphism $\psi : X {\underset G \times} G/H \to Y$ such that the
square 
$$
\CD
X \times G/H @>{p_2}>> X \\
@V{\varpi}VV @V{\pi}VV \\
X {\underset G \times} G/H @>{\psi}>> Y\\
\endCD
$$ 
is cartesian. We define $\varphi$ as the composite map 
$$
\CD 
X \times H/H @>{i}>> X \times G/H @>{\varpi}>> X \times_G G/H 
\endCD,
$$
where $i$ denotes the natural inclusion. Since $i$ yields a section of $p_2$,
we have $\pi = \psi \circ \varphi$. 

To show that $\varphi$ is an $H$-torsor, we may perform an \'etale base change 
on $Y$, and hence assume that $X \simeq G \times Y$ as torsors over $Y$. 
Then $Z \simeq G/H \times Y$ over $Y$, and the statement is obvious.
The remaining assertions are proved along similar lines.
\end{proof}

\subsection{Weak and strong purity}

Given a $G$-scheme $X$ as above, let $\C$ denote a class (or collection) of 
$G$-equivariant $l$-adic local systems $\cL = \{\cL_{\nu}|\nu\}$ on $X$ such 
that the following hypotheses hold: 

\medskip

\noindent
(i) Each local system $\cL$ is mixed and of weight $\ge w$ for some non-negative 
integer $w$.

\medskip

\noindent
 (ii) The class $\C$ is closed under extensions.

\medskip

\noindent
(iii) $\C$ contains the local system $\bQ_l$.

\begin{definition}
\label{def.weak-strong-purity}

{\rm (i)} 
A finite-dimensional $\bQ_l$-vector space $V$ provided with an endomorphism $F$ 
will be called {\it strongly pure}, if each eigenvalue $\alpha$ of $F$ 
(in the algebraic closure $\bar{\bQ}_l$) 
satisfies $\alpha = q^j$ for some integer $j = j(\alpha) \geq 0$.

We will say that the pair $(V,F)$ is {\it weakly pure} if each eigenvalue 
$\alpha$ of $F$ satisfies $\alpha = \zeta q^j$ for some root of unity 
$\zeta = \zeta(\alpha)$ and some integer $j = j(\alpha) \geq 0$.
Equivalently, $\alpha^n = q^{jn}$ for some positive integer $n$.

\medskip
 
{\rm (ii)}
Given a $G$-scheme $X$ and a class of $G$-equivariant local systems $\C$ on $X$
satisfying the above hypotheses, 
we will say that $X$ is {\it weakly (strongly) pure} with respect to the class 
$\C$, if the cohomology space $H^*_{\et}(\bar{X}, \bar {\cL})$, provided with the action
of the Frobenius $F$, is weakly (strongly, \res) pure for each $\cL \in \C$.
 
We will say $X$ is {\it weakly (strongly) pure with respect to $\bQ_l$} 
if the above hypotheses hold for the action of the trivial group and for the 
class generated by the constant $l$-adic local system $\bQ_l$.
\end{definition} 

\begin{remarks}\label{rem:wsp} 

{\rm (i)} The notion of weak purity generalizes that of \cite[Definition 3.2]{BP10},
which is in fact equivalent to the above definition of weak purity for the 
constant local system. However, the notion of strong purity of [loc.~cit.]
differs from the above notion (again, for the constant local system): 
rather than requiring that all eigenvalues of $F$ acting on $H^*_{\et}(\bar{X},\bQ_l)$ 
are integer powers of $q$, it requires that each eigenvalue $\alpha$
of $F$ acting on $H^i_{\et}(\bar{X},\bQ_l)$ satisfies $\alpha^n = q^{in/2}$ for some 
positive integer $n$.

\medskip

{\rm (ii)} The notions of weak and strong purity are stable by extension of 
the base field $\bF_{q}$ to $\bF_{q^n}$ (which replaces the Frobenius endomorphism 
$F$ by $F^n$, and its eigenvalues by their $n$-th powers). 

Specifically, if $X$ is strongly pure over $\bF_{q}$, then so are the base 
extensions $X_{\bF_{q^n}}$ for all $n \geq 1$. Also, $X$ is weakly pure if and only 
if $X_{\bF_{q^n}}$ is weakly pure for some $n$; then $X_{\bF_{q^n}}$ is weakly pure 
for all $n$. In particular, weak purity is independent of the base field $\bF$.

\medskip

{\rm (iii)} For this reason,  {\it we will often write $X$ for $X_{\bar{F}}$ or 
$\bar{X}$, when there is no cause of confusion as to what is intended. Also, 
since we only work with \'etale cohomology, the subscript et will always be 
omitted. As an example, we will denote $H^*_{\et}(\bar{X},\bar{\cL})$ by 
$H^*(X,\cL)$.}

\medskip

{\rm (iv)} If $X$ is smooth and weakly (strongly) pure with respect to $\bQ_l$, then 
so is $H^*_c(X,\bQ_l)$ by Poincar\'e duality. In view of the trace formula
(\ref{eqn:trace}), it follows that the counting function
$n \mapsto \vert X(\bF_{q^n}) \vert$ is a periodic polynomial with integer coefficients
(a polynomial with integer coefficients, \res).

\medskip

{\rm (v)}
Recall that a scheme $X$ is said to be {\it pure}, if all the complex conjugates 
of eigenvalues of $F$ on $H^i(\bar{X}, \bar \bQ_l)$ have absolute value $q^{i/2}$: see 
\cite{De74a, De80}. It may be important to observe that a scheme $X$ may be weakly pure 
without being pure or strongly pure in the above sense: for example, split tori. 
Clearly, for a scheme $X$ that is pure, a necessary condition for strong purity 
is that all the odd $l$-adic cohomology vanishes. Typical examples of schemes that 
are strongly pure are those projective smooth varieties that are stratified by 
affine spaces.

\end{remarks}

We now record some basic properties of these notions:

\begin{lemma}
\label{lem:vec}
{\rm (i)} Let $0 \to V' \to V \to V'' \to 0$ denote a short exact sequence of 
finite-dimensional $\bQ_l$-vector spaces provided with compatible endomorphisms 
$F_{V'},F_V,F_{V''}$. Then $(V,F_V)$ is weakly pure (strongly pure) if and only if 
$(V',F_{V'})$ and $(V'',F_{V''})$ are weakly pure (strongly pure, \res).

\medskip

{\rm (ii)} Given a finite-dimensional $\bQ_l$-vector space $V$ provided with an 
endomorphism $F$, there exists a canonical decomposition 
$V = V_{wp} \oplus V_{wi}$ 
into $F$-stable subspaces, where $V_{wp}$ is the maximal subspace of $V$ 
which is weakly pure. A similar decomposition holds with weakly pure replaced 
by strongly pure.

\medskip

{\rm (iii)} Given pairs $(V,F_V)$ and $(W,F_W)$ that are weakly pure 
(strongly pure), the pair $(V \otimes W, F_V \otimes F_W)$ is also weakly pure 
(strongly pure, \res).
\end{lemma}

\begin{proof} 
For (i), just note that the set of eigenvalues of $F_V$ in $V$ is the union of the
corresponding sets for $(V',F_{V'})$ and $(V'',F_{V''})$.

For (ii), consider the minimal polynomial $P_F$ of the endomorphism $F$, 
that is, the polynomial $P(t) \in \bQ_l[t]$ of smallest degree such that 
$P(F) = 0$ and that $P$ has leading term $1$. Recall that the eigenvalues 
of $F$ are exactly the roots of $P_F$. We have a unique factorization 
$P_F(t) = P_{wp}(t) P_{wi}(t)$ in $\bar{\bQ_l}[t]$, where $P_{wp}(t)$ denotes the 
largest factor of $P_F(t)$ with roots the weakly pure eigenvalues.
Since the set of these eigenvalues is stable under the action of the Galois 
group of $\bar{\bQ}_l$ over $\bQ_l$, we see that both $P_{wp}(t)$ and $P_{wi}(t)$ 
have coefficients  in $\bQ_l$. Also, since $P_{wp}(t)$ and $P_{wi}(t)$ are coprime, 
we have 
$$
V = \Ker\big(P_{wp}(F)\big) \oplus \Ker\big(P_{wi}(F)\big)
$$
which yields the desired decomposition in the weakly pure case.
The strongly pure case is handled by the same argument.
  
Finally, (iii) follows from the fact that the eigenvalues of 
$F_V \otimes F_W$ are exactly the products $\alpha \beta$, where
$\alpha$ is an eigenvalue of $F_V$, and $\beta$ of $F_W$.
\end{proof}

\begin{remark}
In particular, weak and strong purity are stable by negative Tate twists in the following 
sense. Given an integer $n \geq 0$, we let $\Q_l(-n)$ denote the $\bQ_l$-vector space 
of dimension $1$ provided with the action of the Frobenius $F$ by multiplication with 
$q^n$. Let $V$ denote a $\bQ_l$ -vector space provided with an endomorphism $F_V$. 
Then $(V \otimes \bQ_l, F_V \otimes F)$ is weakly pure (strongly pure) if $(V, F)$ is 
weakly pure (strongly pure, \res).
\end{remark}

We now come to the main result of this section:

\begin{theorem} 
\label{thm:main1}
Let $\pi : X \to Y$ denote a torsor for the action of an algebraic group $G$. 
Let $\C_Y$ denote a class of $G$-equivariant $l$-adic local systems on $Y$, and $\C_X$ a 
class of $G$-equivariant $l$-adic local systems on $X$.

(i) Suppose $\C_X \supseteq \pi^*(C_Y)$.
Then if $X$ is  weakly pure with respect to $\C_X$  so is $Y$
with respect to $\C_Y$. In case  $G$ is split, and if  $X$ is strongly pure with 
respect to $\C_X$, then so is $Y$ with respect to $\C_Y$.  

(ii) Suppose $G$ is also connected and $\C_X= \pi^*(\C_Y)$. Then if $Y$ is weakly pure with
respect to $\C_Y$, so is $X$ with respect to $\C_X$. In case $G$ is split and if $Y$ is strongly pure
with respect to $\C_Y$, so is $X$ with respect to $\C_X$.
\end{theorem}

It may be important to point out that the second statement in the above theorem applies 
to the class $\C_X$ of $G$-equivariant $l$-adic local systems on $X$ obtained as 
split summands of $\rho_{X*}({\underline \Q}_l^{\oplus ^n})$ for some $n >0$, where
$\rho_X: X' \ra X$ is a $G$-equivariant finite \'etale map and ${\underline \Q}_l^{\oplus ^n}$ is 
the constant local system of rank $n$ on $X'$, and to the similarly defined class $\C_Y$.
Indeed, taking $X' = X{\underset Y \times}Y'$, the proper base change theorem applied 
to the map $\rho_Y$ shows that 
$\pi^*(\rho_{Y*}({\underline \Q}_l^{\oplus ^n})) = \rho_{X*}({\pi'}^*({\underline \Q}_l^{\oplus ^n})),$ 
where $\pi':X' \ra Y'$ is the map induced by $\pi$. Therefore, 
$\C_X$ contains $\pi^*(\C_Y)$ in this case.

\medskip

The proof of Theorem \ref{thm:main1} consists in decomposing $\pi$ into a sequence of 
torsors under particular groups, and of fibrations with particular fibers, and in
examining each step separately. Specifically, choose a Borel subgroup 
$B \subseteq G$ and a maximal torus $T \subseteq B$, and denote by 
$R_u(G)$ the unipotent radical of $G$. Then $R_u(G) \subset B \subset G^0$.
Moreover, $G' := G^o/R_u(G)$ is a connected reductive group with Borel subgroup 
$B' := B/R_u(G)$ and maximal torus $T' \subseteq B'$, isomorphic to $T$
via the quotient map $G \to G'$. Now $\pi$ factors as the composition of 
the following maps: 

\medskip

\noindent
(i) $\pi_1: X \to X' := X/R_u(G)$, a torsor under $R_u(G)$ 
and hence a locally trivial fibration with fiber an affine space,
\medskip

\noindent
(ii) $\pi_2: X' \to X'/T'$, a torsor under $T' \simeq T$,

\medskip

\noindent
(iii) $\pi_3: X'/T' \to X'/B' \simeq X/B$, a locally trivial fibration
with fiber $B'/T' \simeq R_u(B')$, an affine space again,

\medskip

\noindent
(iv) $\pi_4 : X'/B' \to X'/G' \simeq X/G^o$, a locally trivial fibration
with fiber $G'/B' \simeq G^o/B$, the flag variety of $G'$, 

\medskip

\noindent
(v) $\pi_5 : X/G^o \to X/G$, a Galois cover with group $G/G^o$.

\medskip
 
As in Proposition \ref{prop:equi}, the \'etale fundamental groups of the schemes 
$(X'/T')_{\bar \bF}$, $(X'/B')_{\bar \bF}$ and $(X'/G')_{\bar \bF}$ are all isomorphic,
showing that the $l$-adic local systems on the above schemes correspond 
bijectively. Now let $\C_{X'/T'}$ ($\C_{X'/B'}$, $\C_{X'/G'}$) denote classes 
of $l$-adic local systems on these three schemes that are in bijective 
correspondence after base-extension to $\bar \bF$ under the pull-back maps.

\begin{proposition}
\label{prop:fib}
With the above notation and assumptions, $X'/T'$ is weakly (strongly)
pure with respect to $\C_{X'/T'}$ if and only if so is 
$X'/B'$ is with respect to $\C_{X'/B'}$, if and only if so is
$X'/G'$ with respect to $\C_{X'/G'}$.
\end{proposition}

\begin{proof}
Since $\pi_3$ is a fibration in affine spaces, the first equivalence is clear.   

We will next assume that $X'/B'$ is weakly pure (strongly pure) with respect to 
$\C_{X'/B'}$, and show that $X'/G'$ is weakly pure 
(strongly pure, \res) with respect to $\C_{X'/G'}$. 
Let $\cL' \in \C_{X'/G'}$ and put $\cL := \pi_4^*(\cL')$.  
The Leray spectral sequence
$$
E_2^{s,t} = H^s\big( X'/G', R^t \pi_{4*}(\cL) \big) \Rightarrow H^{s+t}(X'/B', \cL)
$$
and the isomorphism 
$$
R^t\pi_{4*}(\cL) = R^t\pi_{4*}\big( \pi_4^*(\cL') \big) 
\simeq \cL' \otimes R^t\pi_{4*}(\bQ_l)
$$ 
yield a spectral sequence
\begin{equation}\label{eqn:leray}
E_2^{s,t} = H^s\big( X'/G', \cL' \otimes R^t \pi_{4*}(\bQ_l) \big) \Rightarrow 
H^{s+t}\big( X'/B', \pi_4^*(\cL') \big).
\end{equation}
By Deligne's degeneracy criterion (see 
\cite[Proposition 2.4 and (2.6.3)]{De68}), that spectral sequence 
degenerates at $E_2$. In particular, $H^*(X'/G', \cL') =E_2^{s,0}= E_{\infty}^{s,0}$ injects into
$H^s(X'/B', \cL)$, which yields our assertion. 

For the converse, let $\cL \in \C_{X'/B'}$. Without loss of generality, we may 
assume that $\cL = \pi_4^*(\cL')$ for some $\cL' \in \C_{X'/G'}$. Then we still
have the spectral sequence (\ref{eqn:leray}) that degenerates at $E_2$. Thus,
the desired assertion will follow if we show that the $l$-adic sheaf
$R^t \pi_{4*}(\bQ_l)$ is isomorphic to the constant sheaf 
$H^t(G'/B'_{\bar {\mathbb F}}, \Q_l) = \oplus _i\bQ_l(-t_i)$, which is a finite sum with each $t_i \ge 0$. For this we will use 
the Leray-Hirsch theorem (see \cite[Chapter 17, 1.1 Theorem]{Hu94})
adapted to the present framework as follows. 

First the $l$-adic cycle map on the flag variety $G'/B'$ is shown to be an isomorphism in 
\cite[Theorem 5.2]{Jo01}.
The Chow ring of $G'/B'$ is generated over $\bQ_l$ by the Chern classes of the
equivariant line bundles associated with a basis $\{\chi_i| i = 1,\ldots,n\}$
of the character group of $B'$, as shown in \cite[Corollaire 4]{Gr58}. 
We will denote these classes by $\{f_i|i=1, \ldots, n\}$.
Each $\chi_i$ also defines a line bundle over $X'/B'$ via the $B'$-torsor 
$X' \to X'/B'$; this defines Chern classes $\{e_i|i =1, \cdots, n\}$ in
the $l$-adic cohomology of $X'/B'$, which lift the above classes $f_i$
under pull-back to a fiber $G'/B'$.
In this setting, the Leray-Hirsch theorem is the statement that the map
$\alpha \otimes f_i \mapsto \pi_4^*(\alpha)e_i$ defines an isomorphism
$$
H^*(X'/G', \Q_l) \otimes H^*(G'/B', \Q_l) \cong H^*(X'/B', \Q_l).
$$

In case the torsor is locally trivial over a Zariski open covering, 
one may prove the Leray-Hirsch theorem using a Mayer-Vietoris sequence 
with respect to the covering over which the torsor is trivial. 
(See \cite[Chapter 17, 1.1 Theorem]{Hu94} for a proof in the setting of singular 
cohomology, but one can see readily that the proof works also in $l$-adic
cohomology.) Since we are 
considering torsors that are trivial with respect to the \'etale topology, 
this argument needs to be modified as follows. First, let $U \ra X'/G'$ denote
an \'etale covering over which the given torsor ${\mathcal E}=X'/B'$ 
(with fibers $G'/B'$) is trivial. Let $U_{\bullet} =cosk_0^{X'/G'}(U)$ denote 
the hypercovering of $X'/G'$ defined by $U$. If ${\mathcal E}_{U_{\bullet}}$ 
denotes the pull-back of the torsor ${\mathcal E}$ to $U_{\bullet}$, then
${\mathcal E}_{U_{\bullet}}$ trivializes. Therefore, the K\"unneth
formula, along with the hypotheses provides the isomorphism: 
$H^*({\mathcal E}_{U_{\bullet}}, \bQ_l) \simeq 
H^*(U_{\bullet}, \bQ_l) \otimes H^*(G'/B', \bQ_l)$. 
However, since $U_{\bullet}$ is a hypercovering of $X'/G'$ and 
${\mathcal E}_{U_{\bullet}}$ is a hypercovering of $X'/B'$, one obtains 
the identification $H^*(U_{\bullet}, \bQ_l) \simeq H^*(X'/G', \bQ_l)$ 
and $H^*({\mathcal E}_{U_{\bullet}}, \bQ_l) \simeq H^*(X'/B', \bQ_l)$,
see \cite[Proposition 3.7]{Fr83}.

We now show that the sheaves $R^t \pi_{4*}(\bQ_l)$ are of the form as stated.
For this we fix a geometric point $\bar x$ on $X'/G'$ and consider the geometric fiber 
of $\pi_4$ at $\bar x$; this fiber identifies with $G'/B'$. Now we consider the commutative 
diagram:
$$\CD
G'/B' @>{i}>> X'/B' \\
@VVV @V{\pi_4}VV \\
{\bar x} @>>> X'/G' \\
\endCD
$$

This induces a map of the corresponding Leray spectral sequences. Now one obtains the commutative
diagram where the $E_r^{s,t}$-terms of the spectral sequence for the trivial fibration $G'/B' \ra {\bar x}$
are marked as $E_r^{s, t}(2)$:
\vskip .3cm
\xymatrix{{H^t(X'/B', \Q_l)} \twoheadrightarrow \ar@<-1ex>[dd]^{i^*}& {E_{\infty}^{0,t}} & {\cong E_{t+1}^{0, t}}
 & {\simeq  \cdots \simeq} &
{E_2^{0,t} = H^0(X'/G', 
R^t\pi_{4*}(\bQ_l))} \ar@<1ex>[d] \\
&&&& {R^t\pi_{4*}(\bQ_l)_{\bar x} \cong H^t(G'/B', \bQ_l)} \ar@<-1ex>[d]^{id}\\
{H^t(G'/B', \bQ_l)} \twoheadrightarrow & {E_{\infty}^{0,t}(2)} & {\cong E_{t+1}^{0, t}(2)} & {\simeq  \cdots 
\simeq } & {E_2^{0,t}(2) = H^t(G'/B', \bQ_l)}}
\vskip .3cm \noindent
The isomorphisms on the top row follow from the degeneration of the spectral sequence at the $E_2$-terms as
observed above.
Since the spectral sequence for the map $G'/B' \ra {\bar x}$ clearly degenerates with $E_2^{s, t}=0$ for all $s \ne 0$,
the maps in the bottom row (including the first) are all isomorphisms. The naturality of the Leray spectral sequences then shows
that the composition of the maps in the top-row  and the right column identifies with the restriction map 
$i^*:H^t(X'/B', \bQ_l) \ra
H^t(G'/B', \bQ_l)$. Now the map $H^0(X'/G', \bQ_l) \otimes H^t(G'/B', \bQ_l) \ra H^t(X'/B', \bQ_l)$ 
sending the classes $f_i$ to $e_i$ composed with the maps in the top row and the right column of the above 
diagram is an isomorphism. Since the geometric point $\bar x$ was chosen arbitrarily, this shows that 
the corresponding map of the constant sheaf $H^t(G'/B', \bQ_l) $ to the sheaf
$R^t\pi_{4*}(\bQ_l)$ is an isomorphism. Since the $l$-adic cohomology of $G'/B'$ is isomorphic 
to the Chow ring with coefficients in $\bQ_l$ via the cycle map, this yields the decompositions
$H^t(G'/B',\bQ_l) \cong \oplus_i \bQ_l(-t_i)$. This completes the proof of the theorem.
\end{proof}

\begin{remark}
With the above notation and assumptions, the $G$-equivariant $l$-adic 
local systems on $X$ correspond bijectively to the $G/R_u(G)$-equivariant 
local systems on $X'$, by Lemma \ref{lem:desc}. Moreover, if $\C_X$ 
($\C_{X'}$) denote classes of $G$-equivariant local systems on $X$ 
(of $G/R_u(G)$-equivariant local systems on $X'$, \res) that are in 
bijective correspondence (after base extension to $\bar{\bF}$) under 
pull-back, then $X$ is weakly pure (strongly pure) with respect to $\C_X$ 
if and only if $X'$ is weakly pure (strongly pure, \res) with respect 
to $\C_{X'}$. 
\end{remark}

In other words, weak and strong purity are preserved under the above 
maps $\pi_1$, $\pi_3$ and $\pi_4$. The case of the map $\pi_5$ is handled
by the following.

\begin{proposition} 
\label{prop:nonc}
Let $\pi : X \to Y$ denote a torsor under an algebraic group $G$,
and $\psi : Z = X/G^o \to X/G = Y$ the associated Galois cover with 
group $G/G^o$. Let $\C_Y$ ($\C_Z$) denote a class of 
$G$-equivariant $l$-adic local systems on $Y$ ($Z$, \res). 

\medskip

Suppose $\psi_*$ sends $\C_Z$ to $\C_Y$. If $Y$ is weakly pure 
(strongly pure) with respect to $\C_Y$, then so is $Z$ with 
respect to $\C_Z$.

\medskip

Suppose $\psi^*$ sends $\C_Y$ to $\C_Z$. If $Z$ is 
weakly pure (strongly pure) with respect to $\C_Z$, then so is $Y$ 
with respect to $\C_Y$.
\end{proposition}

\begin{proof} 
Once again one considers the fibration 
$G/G^o \to Z {\overset \psi \to} Y$. 
Since $G$ acts trivially on $Y$, observe that 
$EG{\underset G \times}Y \simeq BG \times Y$, so that $G$-equivariant $l$-adic 
local systems on $Y$ correspond to $l$-adic local systems on $Y_{\et}$. 
(Strictly speaking one needs to consider the tensor product 
$V{\underset {\bQ_l} \otimes} \cL$, where $\cL$ is an $l$-adic local system on 
$Y_{\et}$, and $V$ is an $l$-adic representation of the finite group $G/G^o$. 
But $V{\underset {\bQ_l} \otimes} \cL$ is also an $l$-adic local system on $Y_{\et}$. )
 
Suppose that $Y$ is weakly pure with respect to $\C_Y$, 
and $\psi_*$ sends $\C_Z$ to $\C_Y$. Let $\cL \in \C_Z$. 
Then $\psi_*(\cL)$ is a $G$-equivariant $l$-adic local system on $Y$, i.e. 
a local system on $Y$. Moreover, if $\cL$ is mixed of weight $\ge w$, then 
so is $\pi_*(\cL)$. The hypothesis shows that $H^*\big( Y, \psi_*(\cL) \big)$ 
is weakly pure. Recall that one has a spectral sequence 
$E_2^{s, t}= H^s\big( Y, R^t \psi_*(\cL) \big) \Rightarrow H^{s+t}(Z, \cL)$.
This spectral sequence clearly degenerates since $E_2^{s, t}=0$ for $t>0$.
Therefore, the $E_2^{s,t}$ terms identify with the $E_{\infty}^{s,t}$-terms and
it follows that the abutment $H^*(Z, \cL)$ is weakly pure.

Next suppose that $Z$ is weakly pure with respect to 
$\C_Z$, that $\psi^*$ sends $\C_Y$ to $\C_Z$ and that 
$\cL \in \C_Y$. Then $\psi^*(\cL)$ is a $G$-equivariant $l$-adic 
local system on $Z$. This is mixed of weight $\ge w$ if $\cL$ is mixed 
of weight $\ge w$. Therefore, by the hypothesis, 
$H^*\big( Z, \pi^*(\cL) \big )$ is weakly pure. Now $H^*(Y,\cL)$ is a summand of 
$H^*\big( Y, \psi_*(\psi^*(\cL )) \big) = H^*\big( Z, \psi^*(\cL) \big)$ 
so that it is also weakly pure. This completes the proof of the assertions 
about weak purity; those about strong purity are obtained along the same lines.
\end{proof}

Finally, the case of the map $\pi_2$ is handled by an induction on the 
dimension of the torus $T'$ (possibly after a finite extension of fields
so that it becomes split) together with the following:

\begin{proposition} 
\label{prop:torsor}
Let $\pi: X \to Y$ denote a $\bG_m$-torsor. Let $\C_X$ ($\C_Y$) denote 
a class of $\bG_m$-equivariant local systems on $X$ ($Y$ equipped with the trivial
$\bG_m$-action, \res) such that $\pi^*$ sends $\C_Y$  to $\C_X$. 
Then the following hold: (i) if $X$ is weakly pure (strongly pure) with respect to the class
$\C_X$, then so is  $Y$  with respect to the class $\C_Y$. (ii) If in addition, $\C_X = \pi^*(\C_Y)$,
then the converse to (i) also holds.
\end{proposition}

\begin{proof}
Assume that $X$ is weakly pure with respect to $\C_X$, and let
$\cL' \in \C_Y$. Then $\cL := \pi^*(\cL')$ belongs to the class $\C_X$ 
and satisfies
$R^i \pi_*(\cL) \simeq \cL' \otimes R^i\pi_*(\bQ_l)$ for all $i \geq 0$.
Also, since $\bG_m$ is affine and of dimension $1$, Proposition \ref{prop:fibration}
yields that $R^i\pi_*(\bQ_l) = 0$ for $i>1$. Moreover, we have 
$R^i\pi_*(\bQ_l) \simeq \bQ_l$ for $i=0$ and we will presently show that 
$R^1\pi_*(\bQ_l) = \bQ_l(-1)$ so that $R^i\pi_*(\cL) \simeq \cL'$ for $i=0$,  
$R^1\pi_*(\cL) \simeq \cL'(-1)$ and $R^i\pi_*(\cL) = 0$ for $i > 1$. The identification
$\pi_*(\bQ_l) \simeq \bQ_l$ is straightforward; to obtain the identification 
$R^1\pi_*(\bQ_l) \simeq \bQ_l(-1)$, one may proceed as follows. Since 
$R^2\pi_{!}(\bQ_l) \simeq \bQ_l(-1)$ by the trace map, it suffices to show 
using Poincar\'e duality that $R^1\pi_{!}(\bQ_l) \simeq \bQ_l$. 
Let $\pi_1: L \to Y$  denote the line bundle associated with the $\bG_m$-torsor 
$\pi: X \to Y$. Then we have an open immersion $j:X \to L$ with 
complement the zero section $L_0$, isomorphic to $Y$. Denoting that
isomorphism by $\pi_2: L_0 \to Y$, one obtains a distinguished triangle 
$R\pi_!(\bQ_l) \to R\pi_{1!}(\bQ_l) \to R\pi_{2!}(\bQ_l) 
\simeq \bQ_l[0] \to R\pi_!(\bQ_l)[1]$. 
Computing the cohomology sheaves, one obtains the isomorphism 
$\bQ_l \simeq R^1\pi_!(\bQ_l)$ since $R^1\pi_{1!}(\bQ_l) =0$. 

Now the Leray spectral sequence 
$$
E_2^{s, t} = H^s \big( Y, R^t\pi_*(\cL) \big) \Rightarrow H^{s+t}(X, \cL)
$$ 
provides us with the long exact sequence (see, for example, 
\cite[Chapter XV, Theorem 5.11]{CE56}):
\begin{equation}
\label{eqn:exact}
\cdots \to E_2^{n, 0} \to H^n(X, \cL) \to E_2^{n-1, 1} {\overset {d^2} \to} 
E_2^{n+1, 0} \to H^{n+1}(X, \cL) \to E_2^{n, 1} {\overset {d^2} \to} E_2^{n+2, 0} 
\to  \cdots 
\end{equation}
{\it It is important to observe that all the maps in the above long exact sequence are compatible with the
action of the Frobenius.}
In view of the isomorphisms
$$
E_2^{n, 0} = H^n \big( Y, \cL' \otimes \pi_*(\bQ_l) \big) \simeq H^n(Y, \cL'),
$$
$$
E_2^{n, 1} = H^n \big( Y, \cL' \otimes R^1\pi_*(\bQ_l) \big) \simeq H^n(Y, \cL'(-1)) \simeq H^n(Y, \cL')(-1)
$$ 
for all $n \geq 0$, this yields exact sequences (of maps compatible with the Frobenius)
\begin{equation}\label{eqn:gys}
\CD
H^n(X, \cL) @>{\alpha}>> H^{n-1}(Y,\cL')(-1) @>{\beta}>> 
H^{n+1}(Y,\cL') @>{\gamma}>> H^{n+1}(X, \cL)  
\endCD
\end{equation}
for all $n \geq 1$, and an isomorphism
$$
H^0(Y,\cL') \simeq H^0(X,\cL),
$$
all compatible with the Frobenius. 
In particular, $H^0(Y,\cL')$ is weakly pure. We will now argue by
ascending induction on $n$ and assume that we have already shown the
weak purity of $H^i(Y,\cL')$ for all $i \leq n$. Then in the long exact 
sequence (\ref{eqn:gys}), all the terms except for $H^{n+1}(Y,\cL')$ 
are weakly pure. Now one may break up that sequence into the short exact sequence:
$$
0 \to \Ima (\beta) \to H^{n+1}(Y, \cL') \to \Ima (\gamma) \to 0.
$$
Since $\Ima (\gamma)$ is a sub-vector space of $H^{n+1}(X, \cL)$, 
it is weakly pure by Lemma \ref{lem:vec} (i). Moreover, 
$$
\Ima (\beta) \simeq H^{n-1}(Y,\cL')(-1)/\Ker(\beta)
$$ 
and hence $\Ima(\beta)$ is weakly pure by the induction assumption and
Lemma \ref{lem:vec} (i) again. This completes the inductive argument, 
and shows that $Y$ is weakly pure with respect to $\C_Y$.

By the same argument, if $X$ is strongly pure with respect to $\C_X$,
then so is $Y$ with respect to $\C_Y$. 

For the converse direction, one argues similarly by using the exact sequence
$$
H^{n-2}(Y, \cL')(-1)  {\overset  \beta \rightarrow} H^n(Y,\cL') {\overset \gamma \rightarrow} H^n(X, \cL)
 {\overset \alpha \rightarrow} H^{n-1}(Y,\cL')(-1) 
{\overset {\beta} \rightarrow} H^ {n+1}(Y, \cL')
$$
\end{proof}

To conclude this section, we recall a result of Deligne 
(see \cite[9.1.4]{De74b}) that fits very well into our framework, 
and will be a key ingredient of the proof of Theorem \ref{thm:main2}.
\begin{theorem}
\label{thm:BG} 
For any algebraic group $G$, the $l$-adic cohomology $H^*(BG,\bQ_l)$ is 
weakly pure and vanishes in odd degrees.
 
If $G$ is split, then $F$ acts on $H^{2n}(BG,\bQ_l)$ via multiplication by 
$q^n$, for any integer $n \geq 0$.
\end{theorem}

\begin{proof}
We follow the argument sketched in [loc.~cit.]. We first assume that $G$ is split, 
and show the vanishing of cohomology in odd degrees and the second assertion.
Using the fibration $G/G^o \to BG^o \to BG$, one reduces as above to the
case where $G$ is connected. Likewise, using the fibration $G/B \to BB \to BG$, 
and arguing as in the proof of Proposition ~\ref{prop:fib}
one may further assume that $G$ is solvable. Then $G = R_u(G) T$ and 
the fibration $R_u(G) \to BT \to BG$ yields a reduction to the case where $G = T$. 
Since $G$ is split, then $T$ is a direct product of copies of ${\mathbb G}_m$.
Now an induction on the dimension, together with the isomorphism 
$B(T \times T') \simeq BT \times BT'$ (for two tori $T$ and $T'$)
reduce to the case where $T = \bG_m$. In this case, it is well-known that 
${\mathbb P}^{\infty} = \colimn {\mathbb P}^n $ is also a model for $B\bG_m$, 
i.e. the $l$-adic cohomology of the latter is approximated in any given finite 
range by the $l$-adic cohomology of a sufficiently large projective space. 
In view of the structure of $H^*(\bP^n, \bQ_l)$, this yields the last assertion 
and the vanishing in odd degrees.

If $G$ is arbitrary, then it splits over some finite extension ${\mathbb F}_{q^n}$, 
and therefore the arguments above together with Remark \ref{rem:wsp}(ii) yield the  
assertion on weak purity.
\end{proof}

\section{\bf Preservation of the notions of purity by GIT quotients} 

In this section, we show that the properties of weak purity and strong purity 
are preserved by taking certain geometric invariant theory quotients.

We consider a connected reductive group $G$ and a smooth $G$-variety $X$ equipped 
with an ample, $G$-linearized line bundle $L$. We denote by $X^{\sst}$ the open 
$G$-stable subset of $X$ consisting of semi-stable points with respect to $L$,
in the sense of \cite[Definition 1.7]{MFK94}, and by
$$
\pi : X^{\sst} \longrightarrow X^{\sst}/\!/G 
$$ 
the categorical quotient; we will also use the notation $X/\!/G$ for the
GIT quotient variety $X^{\sst}/\!/G$. The complement $X \setminus X^{\sst}$
is the unstable locus of $X$.  The subset of (properly) stable points will be 
denoted by $X^{\st}$; recall from [loc.cit.] that $X^{\st}$ consists of those
points of $X^{\sst}$ having a finite stabilizer and a closed orbit
in $X^{\sst}$. Also, $X^{\st}$ is an open $G$-stable subset of $X$, and 
$\pi$ restricts to a geometric quotient $X^{\st} \to X^{\st}/G$.

\begin{theorem}
\label{thm:main2}
Let $X$ be a smooth $G$-variety equipped with an ample $G$-linearized line bundle
such that the two following conditions are satisfied:

\medskip

\noindent
{\rm (a)} Every semi-stable point of $X$ is stable. 

\medskip

\noindent
{\rm (b)} $X$ admits an  equivariantly perfect stratification with open 
stratum $X^{\sst} = X^{\st}$.

\medskip
 
If $X$ is weakly pure with respect to the constant local system $\bQ_l$, 
then so is the GIT quotient $X/\!/G$. 

\medskip

If $G$ is split and $X$ is strongly pure with respect to $\bQ_l$, then so 
is $X/\!/G$. 
\end{theorem}

\begin{proof} 
By Remark \ref{rem:wsp} (ii), it suffices to show the final assertion.
We begin by obtaining an isomorphism
\begin{equation}\label{eqn:coh.geom.quot}
H^*(X/\!/G, \bQ_l) \simeq H^*(\EGx X^{\sst}, \bQ_l).
\end{equation} 
For this, consider the morphism $\EGx X^{\sst} \to X^{\sst}/G = X /\!/G$, with fiber 
at the image of the geometric point $\bar{x} \in \bar{X}^{\sst}$ being
$\EGx G/G_{\bar{x}} = EG/G_{\bar{x}} = BG_{\bar{x}}$. By  \cite[(A.3.1)]{Jo93}, this yields
a Leray spectral sequence  
$$
E_2^{s, t} = H^s(X/\!/G, R^t\pi_*(\bQ_l)) \Ra H^{s+t}(\EGx X^{\sst}, \bQ_l)
$$ 
together with the identification of the stalks 
$R^t(\pi_*(\bQ_l))_{\bar x} \cong H^t(BG_{\bar x}, \bQ_l)$. Since the stabilizers $G_{\bar x}$ 
are all assumed to be finite groups, it follows that $E_2^{s,t}=0$ for $t>0$ and 
$E_2^{s,0} \cong H^s(X/\!/G, \bQ_l)$; this implies the isomorphism (\ref{eqn:coh.geom.quot}).

Also, one obtains that the pull-back map
$$
H^*(\EGx X, \bQ_l) \longrightarrow H^*(\EGx X^{\sst}, \bQ_l) 
$$
is a surjection by arguing as in \cite[p.~98]{Kir84}. 
Together with the isomorphism (\ref{eqn:coh.geom.quot}) again, 
it follows that it suffices to show that $H^*(\EGx X, \bQ_l)$ is strongly pure. 
For this, we consider the projection $\pi: EG{\underset G \times}X \to BG$
and the associated Leray spectral sequence
$$
E_2^{s,t} = H^s\big( BG, R^t\pi_*(\bQ_l) \big) \Rightarrow H^{s+t}(\EGx X, \bQ_l).
$$ 
Since $G$ is assumed to be connected, $\pi_1(BG_{\et}) {\widehat {}} =0$ 
and hence the $l$-adic local system $R^t\pi_*(\bQ_l) $ is constant. 
(See \cite[(A.8) Theorem]{Jo93} and also \cite[(4.2) Theorem]{Jo02}.) 
Therefore, the $E_2^{s,t}-$term takes on the form 
$H^s(BG, \bQ_l) \otimes H^t(X, \bQ_l)$. Moreover, since $X$ is strongly pure, 
Theorem \ref{thm:BG} and Lemma \ref{lem:vec}(iii) show that 
the $E_2^{s,t}$ terms above are strongly pure, and hence 
so are the $E_{\infty}$ terms as well as the abutment. This completes the proof 
of the theorem. 
\end{proof}

\begin{corollary}\label{cor:git}
Let $G$ be a split connected reductive group and $X$ a smooth projective $G$-variety 
equipped with an ample $G$-linearized line bundle such that every semi-stable point 
is stable. If $X$ is strongly pure, then so is the GIT quotient $X/\!/G$. 
\end{corollary}

\begin{proof}
Let $(S_{\beta})_{\beta \in \bf{B}}$ denote the stratification of $X$ constructed by
by Kirwan (see \cite[Theorem 13.5]{Kir84}). Each $S_{\beta}$ is a smooth, locally
closed $G$-stable subvariety of $X$, and the open stratum $S_0$ equals $X^{\sst}$.
Moreover, 
$$
S_{\beta} = G{\underset{P_{\beta}} \times} Y_{\beta}^{\sst}
$$
for some parabolic subgroup $P_{\beta}$ of $G$ and some smooth, closed 
$P_{\beta}$-stable subvariety $Y_{\beta}^{\sst} \subset S_{\beta}$. Finally, there 
exists a fibration $p_{\beta} : Y_{\beta}^{\sst} \to Z_{\beta}^{\sst}$ with fibers
affine spaces, where $Z_{\beta}^{\sst}$ denotes the subset of semi-stable points 
of a smooth, closed subvariety $Z_{\beta} \subset X$, stable under a Levi subgroup
$L_{\beta}$ of $P_{\beta}$; also, $Z_{\beta}$ is a union of connected components of 
the fixed point locus $X^{T_{\beta}}$, where $T_{\beta}$ is a subtorus of $G$ with
centralizer $L_{\beta}$. 

By the arguments of \cite[Part I]{Kir84}, this stratification is equivariantly
perfect. Specifically, by a criterion of Atiyah and Bott (see \cite[1.4]{AB83}),
it suffices to show that the equivariant Euler class of the normal bundle 
$N_{\beta}$ to $S_{\beta}$ in $X$ is not a zero divisor in $H^*_G(S_{\beta}, \bQ_l)$.
But
$$
H^*_G(S_{\beta},\bQ_l) \cong H^*_{P_{\beta}}(Y_{\beta}^{\sst},\bQ_l) \cong 
H^*_{L_{\beta}}(Y_{\beta}^{\sst},\bQ_l) \cong
H^*_{L_{\beta}}(Z_{\beta}^{\sst},\bQ_l)
$$
and this identifies the equivariant Euler class of $N_{\beta}$ with that of
the restriction $N_{\beta}\vert_{Z_{\beta}^{\sst}}$. But that restriction is
a quotient of the normal bundle $N'_{\beta}$ to $Z_{\beta}^{\sst}$ in $X$,
and the action of $T_{\beta}$ on each fiber of $N'_{\beta}$ has no non-zero fixed vector.
By the lemma below, it follows that the equivariant Euler class of $N'_{\beta}$ is not 
a zero divisor in $H^*_{L_{\beta}}(Z_{\beta}^{\sst},\bQ_l)$; thus, the same holds for the 
equivariant Euler class of $N_{\beta}$.
\end{proof}

\begin{lemma}\label{lem:ab}
Let $L$ be an algebraic group, $Z$ a $L$-variety, and $N$ a $L$-linearized vector
bundle on $Z$. Assume that a subtorus $T$ of $L$ acts trivially on $Z$ and fixes
no non-zero point in each fiber of $N$. Then the equivariant Euler class of $N$
is not a zero divisor in $H^*_L(Z,\bQ_l)$.
\end{lemma}

\begin{proof}
We adapt the argument of \cite[13.4]{AB83}. Choose a maximal torus $T_L$ of $L$
containing $L$. Then the natural map $H^*_L(Z,\bQ_l) \to H^*_{T_L}(Z,\bQ_l)$
is injective; thus, we may replace $L$ with $T_L$, and assume that $L$ is a torus.
Now $L \cong T \times T'$ for some subtorus $T'$ of $L$. Therefore,
$H^*_L(Z,\bQ_l) \cong H^*(BT, \bQ_l) \otimes H^*_{T'}(Z,\bQ_l)$,
since $T$ fixes $Z$ pointwise. Moreover, $N$ decomposes as a direct sum of 
$L$-linearized vector bundles $N_{\chi}$ on which $T$ acts via a non-zero character 
$\chi$. Thus, we may further assume that $N = N_{\chi}$. Then the equivariant 
Euler class of $N$ satisfies $c_d^L(N) =  \prod_{i=1}^d (\chi + \alpha_i)$,
where $d$ denotes the rank of $N$, and $\alpha_i$ its $T'$-equivariant
Chern roots. This is a non-zero divisor in 
$H^*(BT, \bQ_l) \otimes H^*_{T'}(Z,\bQ_l)$ since $\chi \neq 0$.
\end{proof}

Corollary \ref{cor:git} applies for instance to the case where $X$ is a product 
of Grassmannians:
$$
X= \prod_{i=1}^m \Gr(r_i,n), \quad 
L = \boxtimes_{i=1}^m \cO_{\Gr(r_i,n)}(a_i)
$$
where $\Gr(r,n)$ denotes the Grassmannian of $r$-dimensional linear subspaces of 
projective $n$-space, and $\cO_{\Gr(r,n)}(a)$ denotes the $a$-th power of the 
line bundle associated with the Pl\"ucker embedding; here $G = \PGL(n+1)$ and 
$r_1,\ldots,r_m < n$, $a_1,\ldots,a_m$ are positive integers. Indeed, $X$ is 
clearly strongly pure; moreover, $X^{\sst} = X^{\st}$ for general values of 
$a_1,\ldots,a_m$ (see \cite[Section 11.1]{Do03}). The geometric quotient $X/\!/G$ 
is called the space of stable configurations; examples include moduli spaces
of $m$ ordered points in $\bP^n$.

\medskip

Presently we will provide an extension of Corollary \ref{cor:git} to local systems.
We need a general result on lifting group actions under finite \'etale covers.

\begin{proposition}\label{prop:lift}
Let $G$ be a connected algebraic group, $X$ a complete $G$-variety, and $f : X' \to X$ 
a finite \'etale cover. Then there exist a finite \'etale cover of connected algebraic 
groups $\pi: G' \to G$ and an action of $G'$ on $X'$ which lifts the given action of 
$G$ on $X$. If $G$ is reductive, then so is $G'$.
\end{proposition}

\begin{proof} 
Let $\alpha: G \times X \to X$ denote the action map and let 
$F: Z' \to Z = G \times X$ denote the pull-back of $f$ by $\alpha$. 
Then $F$ is again a finite \'etale cover. Since $X$ is complete,  
\cite[Expos\'e X, Corollaire 1.9]{SGA1} shows that the coverings $F_g : X'_g \to X$
are all isomorphic, where $g \in G(k)$, $X$ is identified with the subscheme
$\{g\} \times X$ of $G \times X$, and $F_g$ denotes the pull-back of $F$ to 
$\{g\} \times X$.

Therefore, for each $g \in G(k)$, we have an isomorphism
$X' \to X'_g$ of schemes over $X$. On the other hand, since $\alpha$ is the
action map, then $X'_g$ is the pull-back of $X'$ by the automorphism $g$ of $X$.
Thus, we obtain an isomorphism $X'_g \to X'$  lifting $g : X \to X$.
Composing with the isomorphism $X' \to X'_g$ yields an automorphism of $X'$
which lifts $g$. Therefore, every closed point of $G$ lifts to an automorphism 
of $X'$. 

Now recall that the functor of automorphisms of $X$ is represented by a group 
scheme $\Aut(X)$, locally of finite type. Also, we have a group scheme $\Aut(X',X)$ 
of pairs of compatible automorphisms of $X'$ and $X$, equipped with a homomorphism
$$
f_* : \Aut(X',X) \longrightarrow \Aut(X).
$$
The kernel of $f_*$ is the group scheme of relative automorphisms $\Aut(X'/X)$;
this is a finite reduced group scheme, since its Lie algebra (the derivations of 
$\cO_{X'}$ over $\cO_X$) is trivial. 
On the other hand, we just showed that the image of $f_*$ contains the image of 
$G$ in $\Aut(X)$. We may now take the pull-back of $G \to \Aut(X)$ by $f_*$ 
to obtain a finite \'etale cover $\pi : G' \to G$ such that $G'$ acts on $X'$ 
by lifting the $G$-action on $X$. In case $G'$ is not connected, we replace it 
by its neutral component to obtain the desired cover.

To complete the proof, note that the unipotent radical $R_u(G')$ is a finite cover 
of $R_u(G)$. If $G$ is reductive, then $R_u(G')$ is a finite scheme, and hence a 
point. In other words, $G'$ is reductive. 
\end{proof}

\begin{theorem}\label{thm:main3}
With the assumptions of Corollary \ref{cor:git}, let $\cL =\{ \cL_{\nu}|\nu\}$ denote a 
$G$-equivariant $l$-adic local system on $X$ which satisfies the following two assumptions:

\medskip

\noindent
{\rm (c)} There exists some finite \'etale cover $X'$
of $X$ on which $\cL$ is the constant $G$-equivariant local system (i.e. each $\cL_{\nu}$ is a 
$G$-equivariant constant sheaf).

\medskip

\noindent
{\rm (d)} The restriction of $\cL$ to $X^{\sst}$ is the pull-back of an $l$-adic local system 
$\cM$ on $X/\!/G$.
 
\medskip

If $X$ is weakly pure (strongly pure) with respect to $\cL$, then so is $X/\!/G$ 
with respect to $\cM$.
\end{theorem}

\begin{proof} 
Let $f: X' \to X$ denote the given finite \'etale cover, $L$ the given ample $G$-linearized line bundle
on $X$, and $\pi: G' \to G$ the cover obtained in Proposition \ref{prop:lift}. Then $L' := f^*(L)$ is 
an ample line bundle on $X'$, and we may assume (replacing $G'$ with a further covering) that $L'$ is 
$G'$-linearized. It follows that the strata for the Kirwan stratification of $X'$ (with respect to $G'$ 
and $L'$) are exactly the pull-backs of the strata of $X$. For this stratification, the long exact
sequence in $G'$-equivariant cohomology with respect to the constant local system 
breaks up into short exact sequences.
 
Also, the composition of the maps $\cL \to f_* f^*\cL \to \cL$ of sheaves on $EG{\underset G \times}X$ 
is an isomorphism, where the first map is given by adjunction, and the second one is the trace map. 
In view of assumption (c), it follows that $\cL$ is a direct factor of $f_* (\bQ_l^{\oplus ^n})$. 
Now consider the long exact sequence in $G$-equivariant cohomology with respect to the local system 
$\cL$ provided by the stratification of $X$: the terms in this sequence are summands in the corresponding 
long exact sequence in the $G'$-equivariant cohomology of $X'$. Therefore, this breaks up into short exact 
sequences, and hence the restriction map
\begin{equation}\label{eqn:surj}
H^*_G(X, \cL) \longrightarrow H^*_G(X^{\sst}, \cL)
\end{equation} 
is surjective.

We now adapt the proof of Theorem \ref{thm:main2}. 
Recall the spectral sequence in equivariant cohomology: 
$E_2^{s, t} = H^s \big( BG, R^t\pi_*(\cL) \big) \Ra H^{s+t}_G(X, \cL)$ where 
$\pi: EG{\underset G \times}X \to BG$ denotes the projection. Since 
$G$ is connected, the local system $R^t\pi_*(\cL)$ is constant on $BG$, so that
$E_2^{s,t} = H^s(BG, \bQ_l) \otimes H^t(X, \cL)$. The hypothesis of weak purity 
(strong purity) of $X$ and Lemma~\ref{lem:vec} (iii) show that the $E_2$-terms 
are weakly pure (strongly pure, \res) and hence so are the abutments $H^{s+t}_G(X, \cL)$.
By the surjectivity of the map in (\ref{eqn:surj}), $H_G^*(X^{\sst}, \cL)$ is also
weakly pure (strongly pure, \res). But 
$H^*_G(X^{\sst}, \cL) \cong H^*(X/\!/G, \cM)$ in view of assumption (d).
This completes the proof of the theorem. 
\end{proof}

\begin{remarks}\label{rem:git}
{\rm (i)} The above assumption (d) is satisfied whenever the quotient morphism 
$\pi: X^{\sst} \to X^{\sst}/\!/G$ is a $G$-torsor. Equivalently, the stabilizer of 
any semi-stable point is trivial as a subgroup scheme.

\medskip

\noindent
{\rm (ii)} If the unstable locus $X \setminus X^{\sst}$ has codimension $\geq 2$, then the
fundamental groups of $X$ and $X^{\sst}$ are naturally isomorphic, by 
\cite[Expos\'e X, Corollaire 3.3]{SGA1}. Thus, the local systems on $X$ and $X^{\sst}$ are in 
bijective correspondence via restriction, and the same holds for the $G$-equivariant local
systems. In particular, given any local system $\cM$ on $X/\!/G$, the pull-back
$\pi^*\cM$ extends to a unique $G$-equivariant local system on $X$. 

If in addition $X$ satisfies the assumptions of Corollary \ref{cor:git}, then we see 
that $X/\!/G$ is weakly (strongly) pure with respect to all local systems associated
with finite \'etale covers, if so is $X$.
\end{remarks}

\section{\bf The $l$-adic cohomology of quiver-moduli}

In this section, we begin by recalling some basic facts about quiver representations and the associated
moduli spaces. While the material we discuss is well-known, it seems to be scattered in the 
literature (see e.g. \cite{Kin94, Re08}): we summarize the relevant details from the point of 
view of GIT.

A {\it quiver} $Q$ is a  finite directed graph, possibly with oriented cycles.
That is, $Q$ is given by a finite set of vertices ${ I}$ 
(often also denoted $Q_0$) and a finite set of arrows $Q_1$. The arrows will be 
denoted by $\alpha:i\rightarrow j$. We will denote by 
$\bZ\bfI$ the free abelian group generated by $I$; the basis consisting of elements 
of $I$ will be denoted by $\bfI$. An element $\bfd \in \bZ\bfI$ will be written as 
$\bfd = \sum_{i\in I} d_i \, \bfi $. 

Let ${\rm Mod}(\bF Q)$ denote the abelian category of finite-dimensional representations 
of $Q$ over the finite field $\bF$ (or, equivalently, finite-dimensional representations 
of the path algebra $\bF Q$). Its objects are thus given by tuples 
\be \begin{equation}
\label{quiver.rin}
M=\big( (M_i)_{i\in I},(M_\alpha:M_i\rightarrow M_j)_{\alpha:i\rightarrow j} \big)
\end{equation} \ee

\medskip 

\noindent
of finite-dimensional $\bF$-vector spaces and $\bF$-linear maps between them. 

The {\it dimension vector} $\bfdim(M) \in \bN \bfI$ is defined as 
$\bfdim(M) = \sum_{i\in I}\dim_{\bF }(M_i) \, \bfi$. 
The {\it dimension} of $M$ will be defined to be $\sum_{i \in I} \dim_{\bF }(M_i)$, 
i.e. the sum of the dimensions of the $\bF$-vector spaces $M_i$. This will be denoted 
$\dim(M)$.

We denote by $\Hom_{\bF Q}(M,N)$  the $\bF$-vector space of homomorphisms  
between two representations $M, N\in{\rm Mod}( \bF Q )$.

We will fix a quiver $Q$ and a dimension vector $\bfd = \sum_i d_i \, \bfi$,
and consider the affine space 
$$
X=R(Q, \bfd) :=  \bigoplus_{\alpha:i\rightarrow j} \Hom_{\bF}(\bF^{d_i},\bF^{d_j}).
$$ 
Its points $M=(M_\alpha)_\alpha$ obviously parametrize representations of $Q$ with 
dimension vector $\bfd$. (Strictly speaking only the $\bF$-rational points of $X$ 
define such representations; in general, a point of $X$ over a field extension
$k$ of $\bF$ will define only a representation of $Q$ over $k$ with dimension
vector $\bfd$. We will however, ignore this issue for the most part.) 

The connected reductive algebraic group 
$$
G(Q, \bfd): = \prod_{i\in I} \GL(d_i)
$$ 
acts on $R(Q, \bfd)$ via base change:
$$
\big( (g_i) \cdot (M_\alpha) \big)_\alpha=(g_j M_\alpha g_i^{-1})_{\alpha:i\rightarrow j}.
$$
By definition, the orbits $G(Q, \bfd)\cdot  M$ in 
$R(Q, \bfd)$ correspond bijectively to the isomorphism classes 
$[M]$ of $\bF$-representations of $Q$ of dimension vector $\bfd$. 
We will set for simplicity $G := G(Q,\bfd)$ and $X := R(Q,\bfd)$.
For any $\bar{\bF}$-rational point $M$ of $X$, the stabilizer
$G_M = \Aut_{\bar{\bF}Q}(M)$ is smooth and connected, since it is open in
the affine space $\End_{\bar{\bF}Q}(M)$. Also, note that the
subgroup of $G$ consisting of tuples $(t \id_{d_i})_{i \in I}$, $t \in \bG_m$, 
is a central one-dimensional torus and acts trivially on $X$; moreover, the 
quotient $PG(Q,\bfd)$ by that subgroup acts faithfully.
So one may replace $G$ henceforth by $PG(Q, \bfd)$.

\medskip

We next proceed to consider certain geometric quotients associated to 
the above action. For this, it is important to choose a character of $G$, 
that is, a morphism of algebraic groups $\chi : G \rightarrow \bG_m$. 
If $\O_{X}$ denotes the trivial line bundle on $X$, we will linearize it 
by using the character $\chi^{-1}$: the resulting $G$-linearized line bundle 
on $X$ will be denoted ${\rm L}_{\chi}$. Since this bundle is trivial on
forgetting the $G$-action, a global section of ${\rm L}_{\chi}^n$, for any $n \ge 1$, 
corresponds to $f \otimes _{\bF } t$, where $f \in \bF[X]$,
$t \in \bF [{\mathbb A}^1]$, and $f \otimes_{\bF } t$ denotes their
tensor product over $\bF$.
Now $G$ acts on such pairs $(f, t)$ by $g.(f, t) = (f\circ g, \chi(g)^{-n}t)$ 
where $f\circ g$ denotes the regular function defined by $(f \circ g)(x) = f(gx)$.
Therefore, such a global section will be $G$-invariant precisely when $f$ is 
 {\it $\chi$-semi-invariant with weight $n$}, i.e.  
$$
f(gx)=\chi^n(g)f(x)\mbox{ for all } g\in G \mbox{ and all } x \in X.
$$

\medskip

We will denote the set of all such global sections by $\bF [X]^{G, \chi^n}$.
Therefore, the corresponding geometric quotient (see e.g. 
\cite[Section 8.1]{Do03}) will be defined by
$$
X/\!/G  = \Proj \big( \bigoplus _{n \ge 0} \bF [X]^{G, \chi^n} \big).
$$

\medskip 

Next observe that the only characters of $GL_n$ are powers of the determinant map; 
therefore, the only characters of the group $PG(Q, \bfd)$ are of the form
$$
(g_i)_i \longmapsto \prod_{i\in I}\det(g_i)^{m_i},
$$
for a tuple $(m_i)_{i\in I}$ such that $\sum_{i\in I}m_id_i=0$ to guarantee 
well-definedness on $PG(Q, \bfd)$.

\medskip

Thus, one may choose a linear function 
$\Theta:\bZ \bfI \rightarrow \bZ$ and associate to it a character 
$$
\chi_\Theta((g_i)_i) := \prod_{i\in I} \det(g_i)^{\Theta(\bfd)-\dim (\bfd)\cdot \Theta (\bfi)}
$$
of $PG(Q, \bfd)$. For convenience, we will call $\Theta$ itself a {\it character}. 
(This adjustment of $\Theta$ by a suitable multiple of the function 
$\dim : (d_i) \mapsto \sum_i d_i$ has the advantage that a fixed $\Theta$ can be used 
to formulate stability for arbitrary dimension vectors, and not only those with 
$\Theta(\bfd)=0$. However, this notation is a bit different from the one adopted in 
\cite{Kin94}.)

\medskip

Associated to each character $\Theta$, we define the {\it slope}  $\mu$. This is the function 
defined by $\mu(\bfd) = \frac{\Theta (\bfd)}{\dim (\bfd)}$.
With this framework, one may invoke the usual definitions of geometric invariant theory 
to define the semi-stable points and stable points. Observe that now
a point $x \in R(Q, \bfd)$ will be semi-stable (stable) precisely when there exists a 
$G$-invariant global section of some positive power of the above line bundle that does not 
vanish at $x$ (when, in addition, the orbit of $x$ is closed in the semi-stable locus, 
and the stabilizer at $x$ is finite). Since all stabilizers are smooth and connected, 
the latter condition is equivalent to the stabilizer being trivial.

\medskip

The corresponding varieties of $\Theta$-semi-stable and stable points with respect to the line bundle 
${\rm L}_{\chi}$ will be denoted by
$$ 
R(Q, \bfd)^{\sst} = R(Q, \bfd)^{\Theta-\sst}=  R(Q, \bfd)^{\Theta-\sst}
$$ 
and
$$ 
R(Q, \bfd)^{\st} =  R(Q, \bfd)^{\Theta-\st} =  R(Q, \bfd)^{\Theta-\st}.
$$ 
These are open subvarieties of $X$, possibly empty.
The corresponding quotient varieties will be denoted as follows:
$$
M^{\Theta-\st}(Q, \bfd) = R(Q, \bfd)^{\Theta-\st}/G \mbox{ and } 
M^{\Theta-\sst}(Q, \bfd) = R(Q, \bfd)^{\rm \Theta-\sst}/\!/G = X/\!/G.
$$ 
Observe that 
 the variety $M^{\rm \Theta-\st}(Q, \bfd)$ parametrizes isomorphism classes 
of $\Theta$-stable representations of $Q$ with dimension vector $\bfd$.

\medskip

An orbit $G \cdot M = PG(Q, \bfd) \cdot M$ is closed in $X$ if and only if the 
corresponding representation $M$ is semi-simple, by \cite{Ar69}. The quotient variety 
$X/\!/G$ (for the trivial line bundle with the trivial linearization)
therefore parametrizes isomorphism classes of semi-simple representations of $Q$ 
of dimension vector $\bfd$. It will be denoted by $M^{\rm ssimp}(Q, \bfd)$ and called 
the moduli space of semi-simple representations.

\medskip

We now may state some results taken from \cite[Sections 3 and 4]{Kin94}.

\begin{proposition}\label{prop:king}
With the above notation and assumptions, the following assertions hold:

\begin{itemize}

\item The variety $M^{\rm ssimp}(Q, \bfd)$ is affine.

\item There is a natural projective morphism 
$M^{\Theta-\sst}(Q, \bfd) \rightarrow M^{\rm ssimp}(Q, \bfd)$.
In particular, if the quiver has no oriented cycles, then every $G$-invariant regular
functions on $R(Q, \bfd)$ is constant: therefore, in this case 
$M^{\Theta-\sst}(Q, \bfd)$ is a projective variety.

\item The quotient map $R(Q, \bfd)^{\Theta-\st} \rightarrow M^{\Theta-\st}(Q, \bfd)$ 
is a $PG(Q, \bfd)$-torsor.

\end{itemize}

\end{proposition}

 The following characterization of $\Theta$-(semi-)stable points in $R(Q, \bfd)$ 
is also given in \cite{Kin94}:
 
\begin{theorem} 
A representation $M\in R(Q, \bfd)$ is $\Theta$-semi-stable if and only if 
$\mu(N) \leq \mu(M)$ for all non-zero sub-representations $N$ of $M$. 
The representation $M$ is $\Theta$-stable if and only if $\mu(N) < \mu(M)$ 
for all non-zero proper sub-representations $N$ of $M$.
\end{theorem}

Next we discuss a procedure for determining {\it the instability type} 
of an unstable quiver representation, based on the above theorem. 
One starts with a given representation $M$ of $Q$. Assume that it is {\it unstable}. 
In view of the above proposition, it follows that there is some sub-representation 
$N$ of $M$ for which $\mu(N) > \mu(M)$. We let
$$
{\mathcal U} : = \{N ~\vert~ N \mbox { sub-representation of } M, \; \mu(N) > \mu(M) \}.
$$ 
Let $M^1$ denote a representation in ${\mathcal U}$ such that $\mu(M^1) \ge \mu(N)$ for all 
$N \in {\mathcal U}$ and such that $M^1$ is maximal among all such sub-representations of $M$.
Since $\mu(N) \le \mu(M^1) $ for all sub-representations of $M^1$, it is clear that $M^1$ 
is semi-stable. 

Next we consider the quotient representation $M/M^1$ and apply the above procedure to $M$ 
replaced with $M/M^1$. In case $M/M^1$ is not semi-stable and nonzero, one then obtains a 
sub-representation $\bar M^2$ of $M/M^1$ such that:

\noindent
(i) $\mu({\bar M}^2) \ge \mu({\bar N})$ for any sub-representation $\bar N$ of $M/M^1$, and 

\noindent
(ii) $\bar M^2$ is maximal among  sub-representations $\bar N$ of $M/M^1$ such that
$\mu(\bar N) >\mu (M/M^1)$. 

Let $M^2$ be the sub-representation of $M$ obtained as the inverse image of $\bar M^2$ 
under the quotient map $M \to M/M^1$. Then the choice of $M^1$ shows that $\mu(M^2) < \mu(M^1)$, 
unless $M^1=M$ or $M/M^1$ is semi-stable. Now consider the short exact sequence 
$$
0 \longrightarrow M^1 \longrightarrow  M^2 \longrightarrow M^2/M^1 \longrightarrow 0
$$ 
of  representations. \cite[Lemma 4.1]{Re08} shows that, since $\mu(M^2) <\mu(M^1)$, 
it follows that $\mu(M^2) > \mu(M^2/M^1)$. Combining this with $\mu(M^1) >\mu(M^2)$ shows that 
$\mu(M^1/M^0) = \mu(M^1) > \mu(M^2/M^1)$. Clearly both $M^1/M^0$ and $M^2/M^1$ are semi-stable.

\medskip

One may now {\it repeat the above procedure} to define a finite increasing filtration 
by sub-representations (the so-called {\it Harder-Narasimhan filtration}),
\begin{equation}\label{eqn:fil}
\{0\} =M^0 \subset M^1 \subset M^2 \subset \cdots \subset M^{n-1} \subset M^n = M,
\end{equation}
such that:

\noindent
(i) each $M^i/M^{i-1}$ is semi-stable, and 

\noindent
(ii) $\mu(M^i/M^{i-1}) > \mu (M^j/M^{j-1})$ for all $j>i$. 

Let $\bfd_i $ denote the dimension vector of the representation $M^i/M^{i-1}$. 
Varying $i=1, \cdots, n$, we obtain a sequence $(\bfd^1, \cdots, \bfd^n)$ of dimension vectors.
The sequence of slopes of the sub-quotients given by 
$$
\big( \mu(M^1/M^0),  \cdots, \mu(M^n/M^{n-1}) \big)
$$ 
together with the above sequence of dimension vectors will be called  {\it the instability type} 
of the given unstable representation $M$. Pairs of such sequences
$$
\beta := \big((\mu^1, \cdots, \mu^n), (\bfd^1, \cdots, \bfd^n)\big),$$ 
where $\mu^1 > \mu^2 > \cdots > \mu^n$ is a sequence of rational numbers and the $\bfd^i$ 
are dimension vectors with the properties that $\sum_i \bfd^i = \bfdim(M)$ 
and $\mu^i=\mu(\bfd ^i)$ will be used to index the strata of a natural stratification 
of the representation space $R(Q,\bfd)$. 
(One may observe that there is a slight redundancy in the above data: 
since we assume the character $\Theta$ is fixed once and for all, the choice of the 
dimension vectors $\bfd^i$ determines the slopes $\mu^i$ by the formula 
$\mu^i = \Theta(\bfd^i)/\dim(\bfd^i)$. Nevertheless, we keep the present notation 
for its clarity.)

\medskip
Observe that the closed orbits of $G$ in $R(Q, \bfd)^{\Theta-\sst}$ correspond to the 
$\Theta$-polystable representations, that is, the direct sums of stable 
representations of the same slope. They can also be viewed as the semi-simple 
objects in the abelian subcategory of semi-stable representations of fixed slope $\mu$. 

\medskip

By \cite[Proposition 3.7]{Re03}, the closure relations between the strata may be described 
as follows. For  $\beta = \big( (\mu^1, \cdots, \mu^n), (\bfd^1, \cdots, \bfd^n) \big)$ 
we construct the polygon with vertices, the integral points in the plane defined by 
$(\sum_{i=1}^k \dim(\bfd^i), \sum_{i=1}^k \Theta \big( \bfd^i) \big)$, $k=1, \cdots, n$. 
Then the closure of the stratum indexed by $\beta$ is contained in the union of the strata 
indexed by $\gamma = \big( ( \nu^1, \cdots, \nu^m), (\bfe^1, \cdots, \bfe^m) \big)$, 
when the polygon corresponding to $\gamma$ lies on or above the polygon corresponding to $\beta$.

\medskip

Next we proceed to describe in detail the stratum associated to an index $\beta$ as above, 
adopting the setting of GIT as in \cite{Kir84}. Accordingly the stratum corresponding
to $\beta$, to be denoted $S_{\beta}$,  will be given as follows.

\begin{enumerate}

\item
$S_{\beta}$ is the set of those $M \in X$ admitting a filtration
(\ref{eqn:fil}) such that
$$
\big( (\mu(M^1/M^0), \cdots, \mu(M^n/M^{n-1}) \big), 
\big( \bfdim(M^1/M^{0}), \cdots, \bfdim(M^n/M^{n-1}) \big) = \beta.
$$ 

\item A prescribed filtration (\ref{eqn:fil}) by sub-representations satisfying the above condition,
induces a filtration on each space $M_i$ and therefore a parabolic 
subgroup of $\GL(d_i)$. The  product of these parabolic subgroups is a 
parabolic subgroup $P_{\beta}$ of $G$.

\item We choose a $1$-parameter subgroup $\lambda_{\beta}(t)$ of $G$ 
such that the associated parabolic subgroup is exactly $P_{\beta}$;
in other words,  
$$
P_{\beta} = \{g \in \PGL(Q, \bfd) ~\vert~ 
{\underset {t \to 0} {\lim}}\lambda(t)g\lambda(t^{-1}) 
\mbox{ exists in } \PGL(Q, \bfd) \}.
$$

\item $Y_{\beta}$ is the set of those representations $M \in R(Q,\bfd)$ provided with a chosen filtration
by subrepresentations $M^i$  satisfying the condition that $\mu(M^i/M^{i-1}) > \mu(M^j/M^{j-1})$ for all $j>i$.
(i.e. we do not require the quotients $M^i/M^{i-1}$ to be semi-stable.) $Y^{\sst}_{\beta}$ will denote the
subset of those representations $M$ so that the successive quotients $M^i/M^{i-1}$ of the chosen filtration
are also semi-stable.

\item $Z_{\beta}$ is the set of associated graded modules,
$\gr(M) = \oplus_{i=1}^n M^i/M^{i-1}$, where $M \in Y_{\beta}$. Similarly $Z^{\sst}_{\beta}$ is the set of
associated  graded modules, $\gr(M)$, where $M \in Y^{\sst}_{\beta}$.
\end{enumerate}

\begin{proposition} \label{equiv.perf.0}
{\rm (i)} 
$Z_{\beta} = \{ {\underset {t \to 0} {\lim}} \lambda_{\beta}(t) \cdot M ~|~ M \in Y_{\beta}\}$. 
Similarly, 
$Z^{\sst}_{\beta}= \{ {\underset {t \to 0} {\lim}} \lambda_{\beta}(t) \cdot M ~|~
M \in Y^{\sst}_{\beta}\}$.
\medskip

\noindent
{\rm (ii)} $Z^{\sst}_{\beta}$ is open in the fixed point locus for the action of 
$\lambda_{\beta}(t)$ on $X$.
\end{proposition}

\begin{proof} 
(i) is standard, see \cite[Section 3]{Kin94}. 

(ii) Clearly, the fixed point locus is exactly the set of associated graded modules. 
Moreover, the semi-stability of each $M^i/M^{i-1}$ is an open condition 
(it amounts to semi-stability with respect to the centralizer $L_{\beta}$ of 
$\lambda_{\beta}$, a Levi subgroup of $P_{\beta}$, and to the restriction of
$\Theta$ to $L_{\beta}$).
\end{proof}

\begin{theorem}
\label{thm:main4}
The stratification of $X$ defined above is equivariantly perfect.
\end{theorem}

\begin{proof} Since $X$ and all strata are clearly smooth, it suffices to show 
that the equivariant Euler classes of the normal bundles to the strata $S_{\beta}$ 
are non-zero divisors. But this follows from Proposition \ref{equiv.perf.0}
by arguing as in the proof of Corollary \ref{cor:git}. Specifically, since 
$Z_{\beta}^{\sst}$ is open in the fixed locus $X^{\lambda_{\beta}}$, it is a locally
closed smooth subvariety of $X$, and $\lambda_{\beta}$ has no non-zero fixed point
in the normal spaces to $Z_{\beta}^{\sst}$ in $X$. By Lemma \ref{lem:ab}, it follows that
the Euler class of the normal bundle to $Z_{\beta}^{\sst}$ in $X$ is not a zero divisor. 
\end{proof}

\begin{remark} 
It is important to observe that this result does not follow readily from the theory discussed
in \cite{Kir84}, since one of the key assumptions there is that the variety $X$ be {\it projective}.
The projectivity was needed there, however, only to make sure that the limits as considered 
in the definition of $Z_{\beta}$ exist. Here we prove the existence of these limits directly.
\end{remark}

\begin{corollary} 
\label{cor.quiver.moduli}
Assume in addition to the above situation that each semi-stable point is stable.
Then the $l$-adic cohomology $H^*\big( M^{\Theta-\st}(Q, \bfd),\bQ_l \big)$ 
vanishes in all odd degrees. Moreover, $F$ acts on each 
$H^{2n}\big( M^{\Theta-\st}(Q, \bfd),\bQ_l \big)$ via multiplication by $q^n$. 

In particular, $H^*\big( M^{\Theta-\st}(Q, \bfd),\bQ_l \big)$ is strongly pure,
and hence the number of $\bF_{q^n}$-rational points of $M^{\Theta-\st}(Q, \bfd)$
is a polynomial function of $q^n$ with integer coefficients.
\end{corollary} 

\begin{proof}
We adapt the argument of the proof of Theorem \ref{thm:main2}.
By Theorem \ref{thm:main4}, the pull-back map
$$
H^*(\EGx X, \bQ_l) \longrightarrow H^*(\EGx X^{\sst}, \bQ_l) 
$$
is surjective. The right-hand side is isomorphic to 
$H^*\big( M^{\Theta-\st}(Q, \bfd), \bQ_l \big)$
by our assumption and Proposition \ref{prop:king}. 
On the other hand, since $X$ is an affine space, the left-hand side 
is isomorphic to $H^*(BG, \bQ_l)$. Since $G$ is split, the assertions 
follow from Theorem \ref{thm:BG}.
\end{proof} 

\begin{remarks} 
{\rm (i)} Suppose $\bfd$ is co-prime for $\Theta$, i.e. $\mu(\bfe) \ne \mu(\bfd)$ for all 
$0 \ne \bfe < \bfd$. (For a generic choice of $\Theta$, this is equivalent to  
${\rm g.c.d}\{d_i|i \in I\} = 1$.)
In this case, every semi-stable point is stable.

\medskip 

{\rm (ii)} The above corollary also recovers certain results of Reineke 
(see \cite[Section 6]{Re03} and \cite[Theorem 6.2]{Re06}) which are established 
by using the combinatorics of the Hall algebra associated to the quiver. 
Our proof is purely based on geometric invariant theory coupled with the theory 
of weak and strong purity developed in the first section of this paper.
\end{remarks}

\end{document}